\documentclass[11pt]{amsart}
\usepackage{amsfonts}
\usepackage{amsmath}
\usepackage{amsthm}
\usepackage{bm}
\usepackage[dvips]{graphicx}
\usepackage{natbib}
\usepackage{bigints}
\usepackage{color}

\numberwithin{equation}{section}

\allowdisplaybreaks

\theoremstyle{plain}
\newtheorem{theorem}{Theorem}[section]
\newtheorem{lemma}[theorem]{Lemma}
\newtheorem{proposition}[theorem]{Proposition}

\theoremstyle{definition}

\newtheorem{remark}[theorem]{Remark}
\newtheorem{example}[theorem]{Example}

\def\beqn{\begin{equation}}
\def\beqn*{$$}
\def\eeqn{\end{equation}}
\def\eeqn*{$$}

\newcommand{\BX}{{\bf X}}

\newcommand{\id}{infinitely divisible}
\newcommand{\reals}{{\mathbb R}}
\newcommand{\bbr}{\reals}
\newcommand{\bbn}{{\mathbb N}}

\newcommand{\bbz}{\protect{\mathbb Z}}

\newcommand{\one}{{\bf 1}}

\begin{document}

\bibliographystyle{ecta}

\title[Sample Covariance for Infinitely Divisible Processes]
{Limit Theory for  the Sample Covariance for Heavy Tailed 
  Stationary Infinitely Divisible Processes Generated by Conservative Flows}
\author{Takashi Owada}
\address{School of Operations Research and Information Engineering\\
Cornell University \\
Ithaca, NY 14853}
\email{to68@cornell.edu}

\thanks{This research was partially supported by the ARO
grants W911NF-07-1-0078 and W911NF-12-10385,  NSF grant  DMS-1005903
and NSA grant  H98230-11-1-0154  at Cornell University.}

\keywords{infinitely divisible process, conservative flow, pointwise dual ergodicity,   Darling-Kac theorem 
\vspace{.5ex}}

\begin{abstract}
This study aims to develop the limit theorems on the sample autocovariances and sample autocorrelations for certain stationary infinitely divisible processes. We consider the case where the infinitely divisible process has heavy tail marginals and is generated by a conservative flow. Interestingly, the growth rate of the sample autocovariances is determined by not only heavy tailedness of the marginals but also memory length of the process. Although this feature was first observed by \cite{resnick:samorodnitsky:xue:2000} for some very specific processes, we will propose a more general framework from the viewpoint of infinite ergodic theory. Consequently, the asymptotics of the sample autocovariances can be more comprehensively discussed. 
\end{abstract}

\maketitle

\section{Introduction} \label{sec:intro}

For a discrete strict stationary process $(X_n, \, n \geq 1)$ (i.e., the joint distributions $(X_n, \, n \geq 1)$ and $(X_{n+h}, \, n \geq 1)$ are the same for all positive integers $n_1, \dots, n_k$ and $h$), the sample autocovariance function and the sample autocorrelation function are vital statistics in the analysis of dependence structure of the process. According to the Wold decomposition (see p. 187 in \cite{brockwell:davis:1991}), every strict stationary process with zero mean and finite variance can be represented by the sum of an infinite-order moving average (or equivalently, an ARMA($p,q$) process of finite order) and a perfectly predictable process. Thus, in a classical $L^2$-context, linear models are sufficient for data analysis; indeed, the sample autocorrelation function has traditionally been an important model-fitting and diagnostic tools (see, for example, Chapter 7 of \cite{brockwell:davis:1991}). 

If strict stationary processes lack finite variance, they cannot generally be approximated by linear processes. Thus, it is natural to question whether classical methods based on sample autocorrelations are still plausible. For instance, a major feature of heavy tail models is that the sample autocorrelation converges to a {\it random} limit. If a random limit actually occurs, one needs to be more careful in applying traditional model-fitting and diagnostic tools such as the Akaike Information Criterion or Yule-Walker estimators. For more details, see \cite{davis:resnick:1996}, \cite{resnick:vandenberg:2000} and \cite{resnick:samorodnitsky:xue:1999}. 

To determine the limit behavior of the sample autocovariances of strict stationary processes with infinite variance, it is also important to see how rapidly the sample autocovariances grow. Many studies have revealed that if the tail of a marginal distribution is regularly varying with index $-\alpha$ for some $0<\alpha<2$, then a proper normalizing sequence $(c_n)$ for the sample autocovariances may be written as $c_n=n^{1-2/\alpha}L(n)$, where $L(n)$ is a slowly varying function. Among the processes that possess such type of normalizing sequence are the linear process whose noise distribution has a balanced regularly varying tail (\cite{davis:resnick:1986}), the bilinear process (\cite{davis:resnick:1996}, \cite{resnick:vandenberg:2000}), certain ARCH processes (\cite{davis:mikosch:1998}) and $\alpha$-stable moving average processes (\cite{resnick:samorodnitsky:xue:1999}). 

\cite{resnick:samorodnitsky:xue:2000} reported an interesting phenomenon with respect to the growing rate of the sample autocovariance. They considered a process of the form
\begin{equation}  \label{e:mc.proc}
X_n = \int_{\bbz^{\bbn}} f \circ T^n (x) dM(x)\,,
\end{equation}
where $M$ is a symmetric $\alpha$-stable random measure defined on $(\bbz^{\bbn}, \mathcal{B}(\bbz^{\bbn}))$, and $T(x_0,x_1,\dots) = (x_1,x_2,\dots)$ is the left shift map defined on $\mathbb{Z}^\bbn$. Furthermore, $M$ is assumed to have a control measure of the form $\mu(A) = \sum_{i \in \mathbb{Z}} \pi_i P_i(A)$, where $P_i(\cdot)$ is a probability law of an irreducible, null recurrent Markov chain with state space $\mathbb{Z}$, and $( \pi_i )$ is its unique (up to multiplicative factors) $\sigma$-finite and invariant measure. By introducing an extra parameter $0 \leq \beta \leq 1$, they proved that a proper normalizing sequence in this situation is $c_n = n^{(1-\beta)(1-\alpha/2)} L(n)$. 
The parameter $\beta$ accounts for the significantly longer memory of this process, relative to the other processes described in the previous paragraph; more details can be found in \cite{samorodnitsky:2005}. 

An obvious drawback of the process given in \eqref{e:mc.proc} is the highly specific form of the process and its control measure. In this paper, we propose a more general framework inspired by the infinite ergodic theory, in which the asymptotics of the sample autocovariances can be more comprehensively assessed. In terms of the growth rate of the sample autocovariance and its weak limit, we will demonstrate that results similar to those of \cite{resnick:samorodnitsky:xue:2000} are obtainable in the generalized framework. 

In Section \ref{sec:ergodic.theory}, we will overview the basic concepts of infinite ergodic theory that applies in this paper. Section \ref{sec:main} is the main body of this paper and provides the limit theorems on the sample autocovariances and the sample autocorrelations for the process of our interest. All supplemental results necessary for the proof are collected in Section 4, and exploiting these results, Section 5 completes the proof. This paper also covers three examples: the first one treats once again the process in \eqref{e:mc.proc} under the generalized framework. The other two examples are related to certain ergodic dynamical systems depicted by the so-called basic AFN-map and $S$-unimodal map, respectively. 

Given a strict stationary process $\BX = (X_n, \, n \geq 1)$, the sample autocovariance is denoted by 
$$
\widehat{\gamma}_n(h) = \frac{1}{n} \sum_{k=1}^n X_k X_{k+h}, \ \ \ h=0,1,2,\dots,
$$
and the sample autocorrelation function by $\widehat{\rho}_n(h) = \widehat{\gamma}_n(h) / \widehat{\gamma}_n(0)$, $h=0,1,2,\dots$. Given a measure space $(E,\mathcal{E},\mu)$ on which an operator $T:E \to E$ is defined, a partial sum related to a measurable function $h: E \to \bbr$ is written as
$$
S_n(h)(x) = \sum_{k=1}^n h \circ T^k(x)\,, \ \ x \in E\,.
$$
Throughout the paper, the convergence $\Rightarrow$ means weak convergence, and $RV_{\gamma}$ represents a family of regularly varying functions with exponent $\gamma \in \bbr$. Since our interest always lies in a discrete strict stationary process, we simply call it a stationary process.

\section{Ergodic Theoretical Notions}  \label{sec:ergodic.theory}

In this section, we will present the basic notions on ergodic theory used in the sequel. For further studies, the main references are \cite{krengel:1985}, \cite{aaronson:1997}, and \cite{zweimuller:2009}. 

Let $(E,\mathcal{E},\mu)$ be a $\sigma$-finite, infinite measure space. We will often denote $A = B$ mod $\mu$ for $A, B \in \mathcal{E}$ when $\mu(A\triangle B)=0$. 

Let $T:E \to E$ be a measurable map. $T$ is called {\it ergodic} if any (almost) invariant set $A$ with respect to $T$ (i.e., $T^{-1}A = A$ mod $\mu$) satisfies $\mu(A)=0$ or $\mu(A^c)=0$.

The map $T$ is said to be {\it conservative} if 
$$
\sum_{n=1}^{\infty} \one_A \circ T^n = \infty \ \ \text{a.e. on } A
$$
for any $A \in \mathcal{E}$, $0 < \mu(A) < \infty$. When the whole sequence $(T^n)$ gets involved, $(T^n)$ is particularly called a {\it flow}. 

In view of the Hopf decomposition (see e.g., Theorem 3.2 in \cite{krengel:1985}), any state space $E$ can be partitioned into two measurable invariant subsets $C$ and $D$, such that the map $T$ is conservative on $C$ and $D=E \setminus C$. We usually refer to $C$ and $D$ as a conservative part and a dissipative part, respectively. From its definition, $C$ is viewed as a set such that, departing from an arbitrary $A \subseteq C$, one could keep coming back to $A$ infinitely often. On the contrary, even if starting from $A \subseteq D$, one may not come back to $A$ quite often. 

Next we define a {\it dual operator} $\widehat{T}: L^1(\mu) \to L^1(\mu)$ by 
$$
\widehat{T} f = \frac{d \bigl(\nu_f \circ T^{-1}\bigr)}{d \mu}\,,
$$
where $\nu_f$ is a signed measure defined by $\nu_f(A) = \int_{A} f d\mu$, $A \in \mathcal{E}$. It is worth providing the dual relation
\begin{equation} \label{e:dual.rel}
\int_E   \widehat{T} f\cdot g\, d\mu = \int_E f\cdot g \circ T\, d\mu
\end{equation}
for $f\in L^1(\mu), \, g\in L^\infty(\mu)$. Note that, for any nonnegative measurable function $f$ on $E$, a similar definition gives a nonnegative measurable function $\widehat{T} f$, and that \eqref{e:dual.rel} holds for any two nonnegative measurable functions $f$ and $g$. 

A conservative ergodic and measure preserving map $T$ is said to be {\it pointwise dual ergodic}, if there exists a normalizing sequence $a_n \nearrow \infty$ such that as $n \to \infty$,
\begin{equation}  \label{e:pde}
\frac{1}{a_n} \sum_{k=1}^n \widehat{T}^k f \to \mu(f) \ \ \ \text{a.e.} \; \text{for every } f \in L^1(\mu)\,. 
\end{equation}
We often require that the above convergence takes place uniformly on a set of finite measure. Let $A \in \mathcal{E}$ with $0 < \mu(A) < \infty$. $A$ is said to be a uniform set for a conservative ergodic and measure preserving map $T$, if there exist a normalizing sequence $a_n \nearrow \infty$ and a nonnegative measurable function $f \in L^1(\mu)$ such that as $n \to \infty$,
\begin{equation}  \label{e:uniform_set}
\frac{1}{a_n} \sum_{k=1}^n \widehat{T}^k f \to \mu(f) \ \ \ \text{uniformly, a.e. on } A\,. 
\end{equation}
If a measurable function $f$ in \eqref{e:uniform_set} can be replaced by an indicator function $\one_A$, the set $A$ is particularly called {\it a Darling-Kac set}. From the similar argument as the proof of Proposition 3.7.5 in \cite{aaronson:1997}, one can see that $T$ is pointwise dual ergodic if and only if $T$ admits a uniform set. It is important to note that it is legitimate to use the same sequence $(a_n)$ both in (\ref{e:pde}) and (\ref{e:uniform_set}).

We often need to put a more strict assumption than \eqref{e:uniform_set}. Let $A \in \mathcal{E}$ with $0 < \mu(A) < \infty$. $A$ is said to be {\it a uniformly returning set} for a conservative ergodic and measure preserving map $T$, if there exist a normalizing sequence $b_n \nearrow \infty$ and a nonnegative measurable function $f \in L^1(\mu)$ such that as $n \to \infty$,
\begin{equation}  \label{e:def_uniformly_returning_set}
b_n \widehat{T}^n f \to \mu(f) \ \ \ \text{uniformly, a.e. on } A\,. 
\end{equation}
Clearly any uniformly returning set is a uniform set. Further information on uniformly returning sets is given, for example, in \cite{kessebohmer:slassi:2007}.

Given a uniform set (or a Darling-Kac set or a uniformly returning set) $A$, a natural question is how often the set $A$ will be visited as we evolve along the flow $( T^n )$. Such frequency is usually measured by a wandering rate 
\begin{equation}  \label{e:wanderingrate}
w_n = \mu \Bigl(\bigcup_{k=0}^{n-1} T^{-k}A\Bigr)\,. 
\end{equation}
There are some other alternative expressions for \eqref{e:wanderingrate}. To get those alternatives, we define the first entrance time to $A$
$$
\varphi(x) = \min \{ n \geq 1: T^n x \in A \}\,.
$$
(Notice that $\varphi < \infty$ a.e. on $E$, if $T$ is conservative ergodic and measure preserving.) It is elementary to prove that $\mu(A \cap \{ \varphi > k \}) = \mu(A^c \cap \{ \varphi = k \})$, $k \geq 1$. Therefore, we get 
$$
w_n = \mu(A) + \sum_{k=1}^{n-1} \mu (A^c \cap \{ \varphi = k  \}) = \sum_{k=0}^{n-1} \mu(A \cap \{ \varphi > k \})\,.
$$
This, in turn, implies
\begin{equation}  \label{e:equivalent_wandering}
w_n \sim \mu(\varphi < n) \ \ \ \text{as } n  \to \infty\,. 
\end{equation}

Let $T$ be a pointwise dual ergodic map and let $A$ be a uniform set determined by $T$. Then there is a precise connection between the return sequence $(w_n)$ and the normalizing sequence
$(a_n)$ in \eqref{e:uniform_set} (and, hence, also in \eqref{e:pde}), if regular variation is assumed. In fact, if either $(w_n) \in RV_{1-\beta}$ or $(a_n) \in RV_{\beta}$ for some
$\beta \in [0,1]$, then  
\begin{equation}  \label{e:prop3.8.7} 
a_n \sim \frac{1}{\Gamma(2-\beta) \Gamma(1+\beta)} \frac{n}{w_n} \ \ \
\text{as } n \to \infty\,. 
\end{equation}
Indeed, Proposition 7 in \cite{zweimuller:2009} gives the proof of this statement when $T$ has a Darling-Kac set, but this can be easily generalized to the case when $T$ is a pointwise dual ergodic map. 

Analogously, a similar kind of connection between $(w_n)$ and $(b_n)$ in \eqref{e:def_uniformly_returning_set} was shown by \cite{kessebohmer:slassi:2007}. If either $(w_n) \in RV_{1-\beta}$ or $(b_n) \in RV_{1-\beta}$ for some $\beta \in (0,1]$, then
\begin{equation}  \label{e:bn_wn}
b_n \sim \Gamma(\beta) \Gamma(2-\beta) w_n \ \ \ \text{as } n \to \infty\,.
\end{equation}

\section{Limit Theorem on the Sample Autocovariances} \label{sec:main}

This section presents the main limit theorem on the sample autocovariances and the sample autocorrelations for heavy-tailed infinitely divisible processes, which will be rigorously formulated soon.  
We consider an infinitely divisible process
\begin{equation}  \label{e:def_Xn}
X_n = \int_E f \circ T^n(x) dM(x)\,, \ \ \ n=1,2,\dots,
\end{equation}
where $M$ is an independently scattered \id\ random measure on a measurable space $(E,\mathcal{E})$. The random measure $M$ is assumed to be homogeneous symmetric and have a local L\'evy measure $\rho$ and a $\sigma$-finite {\it infinite} control measure $\mu$. We assume, throughout the paper, that a Gaussian component is identically zero. By these assumptions on the random measure $M$, we may write, for every $A \in \mathcal{E}$ of finite $\mu$-measure,
$$
E e^{iu M(A)} = \exp\left\{ -\mu(A) \int_{\mathbb{R}} \bigl(1-\cos(ux)\bigr)\, \rho(dx) \right\}\, \ \ \ u \in \bbr.
$$

One of the central assumptions in our work is the heavy tailedness of the process $\BX=(X_1,X_2,\dots)$. We assume that $\rho$ has a regularly varying tail with index $-\alpha$, $0 < \alpha < 2$:
\begin{equation}  \label{e:RV_Levy}
\rho(\cdot, \infty) \in RV_{-\alpha} \ \ \text{at infinity.}
\end{equation}
In what follows, in order to emphasize the dependence on the tail parameter $\alpha$, we express $\rho$ by $\rho_{\alpha}$. 
An extra assumption will be added on the lower tail of $\rho_{\alpha}$: for some $p_0 \in (0,2)$, 
\begin{equation}  \label{e:lower_tail_Levy}
x^{p_0} \rho_{\alpha}(x,\infty) \to 0 \ \ \text{as } x \downarrow 0\,.
\end{equation}

The other crucial assumption is that the process $\BX$ is generated by a conservative flow. A conservative flow is known to be related to long memory in the process $\BX$; the length of memory observed in $\BX$ is significantly longer than that in the process generated by a dissipative flow (e.g. $\alpha$-stable moving average processes). See for example, \cite{samorodnitsky:2004} and \cite{roy:2008}. With this concept in view, let $T$ be a conservative ergodic and measure preserving map on $(E,\mathcal{E},\mu)$. Furthermore, $T$ is pointwise dual ergodic, and hence, $T$ admits some uniform set $A \in \mathcal{E}$ with $0 < \mu(A) < \infty$. We suppose that the normalizing sequence $(a_n)$ for pointwise dual ergodicity is regularly varying with exponent $0 \leq \beta < 1$. A certain pointwise dual ergodic map $T$ (e.g., Markov shift operator; see Section 4.5 in \cite{aaronson:1997}) is known to satisfy
$$
\widehat{T}^k \one_{A^c \cap \{ \varphi = k \} }(x) = \mu \bigl( A^c \cap \{ \varphi = k \} \bigr), \ \ \ \text{for all } x \in A, \ \ k \geq 1\,,
$$
where $\varphi(x) = \min \{ n \geq 1: T^n x \in A \}$ is the first entrance time to $A$. 
We generalize this condition by assuming that 
\begin{equation}  \label{e:generalize_Markov_shift}
\frac{1}{\mu(\varphi \leq n)} \sum_{k=1}^n \widehat{T}^k \one_{A^c \cap \{\varphi = k \} } (x) \ \ \text{is uniformly bounded on } A\,.
\end{equation}

Let $f:E \to \bbr$ be a measurable function whose support is contained in $A$, that is, supp$(f) \subset A$. We often write $f_n(x) = f \circ T^n(x)$, $x \in E$. Moreover, we put the integrability condition 
\begin{equation}  \label{e:integrability}
f \in L^2(\mu) \ \ \text{with } \ \ \mu(f^2) = \int_E f(x)^2 \mu(dx) >0\,. 
\end{equation}

Now, the process $\BX = (X_1,X_2,\dots)$ in \eqref{e:def_Xn} turns out to be a well-defined stationary \id\ process with L\'evy measure of $X_n$ given by 
$$
(\rho_{\alpha} \times \mu)\{(x,s): xf_n(s) \in \cdot  \}\,.
$$ 
See \cite{rajput:rosinski:1989} for further information on spectral representations of infinitely divisible processes. 
We can see from \cite{rosinski:samorodnitsky:1993} that the tail of $X_n$ is asymptotically the same as the tail of its L\'evy measure. That is,
$$
P(X_n>\lambda)\sim \Bigl( \int_E |f(s)|^\alpha\,
\mu(ds)\Bigr) \rho_{\alpha} (\lambda,\infty) \  \ \ \text{as } \lambda\to\infty\,.
$$
Due to regular variation of $\rho_{\alpha}$, this implies that the process $\BX$ belongs to the domain of attraction of a symmetric $\alpha$-stable law.

The argument for the proof of our main limit theorem will be separated into two cases. First, we discuss the case when $\alpha$ and $\beta$ lie in the range 
\begin{equation}  \label{e:first_range}
\text{either } 1 < \alpha < 2, \, 0 \leq \beta < 1 \ \ \text{or } \ 0 < \alpha \leq 1, \, 0 \leq \beta < 1/(2-\alpha)\,.
\end{equation}
As compared with this parameter range, the complement of \eqref{e:first_range}, $0 <\alpha \leq 1$ and $1/(2-\alpha) \leq \beta < 1$, is unfortunately more difficult to handle. 
One of the possible ways to overcome this difficulty is to assume that the product map $T \times T$ still has {\it nice properties}. Namely, we could assume that $T \times T$ is still a pointwise dual ergodic map on a measure space $(E \times E, \mathcal{E} \times \mathcal{E}, \mu \times \mu)$. However, as will be studied later, $T \times T$ is not necessarily pointwise dual ergodic. In that case, alternatively, we could put more stringent assumption on the set $A$; that is, $A$ is assumed to be a uniformly returning set for $\one_A$, i.e., there exists an increasing normalizing sequence $(b_n)$ such that as $n \to \infty$,
\begin{equation}  \label{e:uniformly_returning_set}
b_n \widehat{T}^n \one_A \to \mu(A) \ \ \ \textrm{uniformly, a.e. on } A\,. 
\end{equation}

Before stating the main limit theorem, we would like to determine a normalizing sequence $(c_n)$ which enables us to capture how rapidly the sample autocovariances of the process $\BX$ grow. To this end, we need some preliminary work. 
For $0 < \beta < 1$, let $( S_{\beta}(t), t \geq 0 )$ be a $\beta$-stable subordinator, i.e., a stable L\'evy process with increasing sample path. Assume that the moment generating function of $(S_{\beta}(t))$ is given by $E \exp \{-\theta S_{\beta}(t) \} = \exp \{ -t \theta^{\beta} \}$ for $\theta > 0$ and $t \geq 0$. Define its inverse process by
\begin{equation}  \label{e:MLprocess}
M_{\beta}(t) = S_{\beta}^{\leftarrow}(t) =  \inf \{u \geq 0: S_{\beta}(u) \geq t  \}, \ \ \ t \geq 0\,. 
\end{equation}
The process $( M_{\beta}(t) )$ is called the Mittag-Leffler process with index $\beta$ because the moment generating function of $( M_{\beta}(t) )$ is given by the Mittag-Leffler function
\begin{equation}  \label{e:mgf_Mittag_Leffler}
E \exp \{ \theta M_{\beta}(t) \} = \sum_{n=0}^{\infty} \frac{(\theta t^{\beta})^n }{\Gamma(1+n \beta)}, \ \ \ \theta \in \mathbb{R}; 
\end{equation}
see Proposition 1(a) in \cite{bingham:1971}. 
The Mittag-Leffler process is well-defined in the limiting case $\beta=0$ as well. It follows from expression \eqref{e:mgf_Mittag_Leffler} that $M_0(t) \equiv M_0$ can be regarded as an exponential random variable of unit parameter. 

In addition, for later use, let $V_{\beta}$, $0 \leq \beta < 1$ denote a random variable with density
\begin{equation}  \label{e:pdfTinf}
g_{V_{\beta}} (x) = (1-\beta) x^{-\beta}, \ \ \ 0 < x \leq 1\,.
\end{equation}
Here, $V_{\beta}$ is taken to be independent of $( M_{\beta}(t) )$. 

Let $U_{\alpha}(x) = \rho_{\alpha}(x,\infty)$, $x > 0$. Define the right continuous inverse of $U_{\alpha}(x)$ by
$$
U_{\alpha}^{\leftarrow}(y) = \inf\{x>0: U_{\alpha}(x) \leq y  \}\,, \ \ \ y>0\,.
$$
Given the normalizing sequence $(a_n)$ for pointwise dual ergodicity and its wandering rate sequence $(w_n)$, we define
\begin{equation}  \label{e:normalizing_const}
c_n = 2^{2/\alpha} C_{\alpha,\beta} C_{\alpha/2}^{-2/\alpha} a_n \bigl( U_{\alpha}^{\leftarrow} (w_n^{-1}) \bigr)^2\,,
\end{equation}
where 
\begin{equation}  \label{e:def.C_al.be}
C_{\alpha,\beta} = \Gamma(1+\beta) \bigl(E M_{\beta}(1-V_{\beta})^{\alpha/2} \bigr)^{2/\alpha}
\end{equation}
and $C_{\alpha/2}$ is a tail constant of an $\alpha/2$-stable random variable; see \cite{samorodnitsky:taqqu:1994}. By the definition of $(c_n)$, one can directly calculate its regular variation exponent:
\begin{equation}  \label{e:RV_index_cn}
c_n \in RV_{\beta+2(1-\beta)/\alpha}\,.
\end{equation}
Therefore, the growth rate of the sample autocovariance of the process $\BX$ is determined by not only heavy tailedness of the marginals (through $\alpha$) but also the length of memory (through $\beta$). This is in contrast to the case of the processes generated by dissipative flows, e.g., $\alpha$-stable moving averages studied by \cite{resnick:samorodnitsky:xue:1999}. Indeed, it was shown in \cite{resnick:samorodnitsky:xue:1999} that the sample autocovariances of the $\alpha$-stable moving averages grow at a regularly varying rate with exponent $2/\alpha$, regardless of the length of memory. A substitution of $\beta=0$ into (\ref{e:RV_index_cn}) yields $c_n \in RV_{2/\alpha}$, which implies that $\beta=0$ corresponds to the shortest memory in the process $\BX$. As $\beta$ gets closer to $1$, it is expected to exhibit longer memory. 

Finally we need to recall a few useful representations for the process $\BX$. First, we decompose the process $\BX$ by the magnitude of the L\'{e}vy jumps. Let
$$
\rho_{\alpha,1}(\cdot) = \rho_{\alpha}(\cdot \cap \{ x: |x| > 1 \})\,,
$$
$$ 
\rho_{\alpha,2}(\cdot) = \rho_{\alpha}(\cdot \cap \{ x: |x| \leq 1 \})\,,
$$
and let $M_i$, $i=1,2$, denote homogeneous symmetric \id\ random measures with the same control measure $\mu$ and local L\'{e}vy measures $\rho_{\alpha,i}$, $i=1,2$. 
Then, $X_n$ can be written as 
$$
X_n \stackrel{d}{=} \int_E f_n (x) dM_1(x) + \int_E f_n (x) dM_2(x)\,.
$$
Denote $X_n^{(i)} = \int_E f_n(x) dM_i(x)$, $i=1,2$. We may write
\begin{equation}  \label{e:decomp.4parts}
\frac{n}{c_n} \widehat{\gamma}_n(h)  
\stackrel{d}{=} c_n^{-1} \left( \sum_{k=1}^n X_k^{(1)} X_{k+h}^{(1)} + \sum_{k=1}^n X_k^{(1)} X_{k+h}^{(2)} + \sum_{k=1}^n X_k^{(2)} X_{k+h}^{(1)} + \sum_{k=1}^n X_k^{(2)} X_{k+h}^{(2)} \right)\,.
\end{equation}
We also recall a certain series representation of $(X_n)$, which was developed by \cite{rosinski:1990}. The reader may refer to Section 3.10 in \cite{samorodnitsky:taqqu:1994} as well. Since $\mu$ is a $\sigma$-finite measure, one can find a $\mu$-equivalent probability measure $\mu_0$ such that 
$$
\mu_0 (B) = \int_B q(x) \mu(dx)\,,  
$$
where $q$ is a positive measurable function on $E$. For $l=1,2$, we write $U_{\alpha,l}(x) = \rho_{\alpha,l}(x,\infty)$ for $x > 0$ and define the right continuous inverse of $U_{\alpha,l}(x)$ by
$$
U_{\alpha,l}^{\leftarrow} (y) = \inf \{x> 0: U_{\alpha,l}(x) \leq y \}\,, \ \ \ y > 0\,.  
$$
According to \cite{rosinski:1990}, $X_n^{(l)}$ can be represented in law as 
\begin{equation}  \label{e:series.representation}
(X_n^{(l)},\, n \geq 0 ) \stackrel{d}{=} \left( \sum_{i=1}^{\infty} \epsilon_i U_{\alpha, l}^{\leftarrow} \left(  \frac{\Gamma_i q(V_i)}{2} \right) f_n(V_i),\, n \geq 0 \, \right)\,,
\end{equation}
where $( \epsilon_i )$ is an i.i.d. Rademacher sequence taking $1$ or $-1$ with probability $1/2$, $\Gamma_i$ is the $i$th jump time of a unit rate Poisson process, and $( V_i )$ is a sequence of i.i.d. random variables with common distribution $\mu_0$. 

\begin{theorem} \label{t:main_theorem1}
Let $M$ be a symmetric homogeneous infinitely divisible random measure on $(E,\mathcal{E})$ with control measure $\mu$ and local L\'evy measure $\rho_{\alpha}$, which satisfies (\ref{e:RV_Levy}) and (\ref{e:lower_tail_Levy}). \\
Let $T$ be a pointwise dual ergodic map on a $\sigma$-finite infinite measure space $(E,\mathcal{E},\mu)$ with normalizing sequence $(a_n) \in RV_{\beta}$. Suppose that $T$ admits a uniform set $A \in \mathcal{E}$, $0 < \mu(A) < \infty$, and (\ref{e:generalize_Markov_shift}) is fulfilled. \\
Let $f:E \to \bbr$ be a measurable function that is supported by $A$ and satisfies integrability condition (\ref{e:integrability}). \\
If $\alpha$ and $\beta$ lie in the range (\ref{e:first_range}), then the stationary infinitely divisible process $\BX$ given in (\ref{e:def_Xn}) satisfies for $H \geq 0$, 
\begin{equation}  \label{e:main_result}
\left( \frac{n}{c_n} \widehat{\gamma}_n(h), \, h=0,\dots,H \right) \Rightarrow \bigl( \mu(f \cdot f_h) W,\, h=0, \dots, H \bigr) \ \ \ \text{in } \bbr^{H+1}
\end{equation}
as $n \to \infty$. Here, $(c_n)$ is given in \eqref{e:normalizing_const} and $W$ is a positive strictly stable random variable of exponent $\alpha / 2$, i.e., the characteristic function of $W$ is 
\begin{equation}  \label{e:pos_stable_subordinator}
E e^{iuW} = \exp \{ \int_{(0,\infty)} \hspace{-5pt} (e^{iux} - 1) \rho_{*}(dx) \} \ \ \ u \in \bbr\,,
\end{equation}
with $\rho_{*}(dx) = 2^{-1}\alpha C_{\alpha/2} x^{-1-\alpha/2} dx$, $x>0$. \\
As a consequence, we also get
\begin{equation}  \label{e:sub_main_result}
\widehat{\rho}_n(h) \stackrel{p}{\to} \frac{\mu(f \cdot f_h)}{\mu(f^2)} \ \ \ \text{as } n \to \infty\,.
\end{equation}

On the other hand, if $\alpha$ and $\beta$ lie outside of the range (\ref{e:first_range}), we additionally suppose either $(i)$ or $(ii)$ below: \\
(i): $T \times T$ is still a pointwise dual ergodic map on $(E \times E, \mathcal{E} \times \mathcal{E}, \mu \times \mu)$ with normalizing sequence $(a_n^{\prime}) \in RV_{2\beta-1}$, and further, we extend condition (\ref{e:generalize_Markov_shift}) to a two-dimensional version:
\begin{equation}  \label{e:2dim_generalize_Markov_shift}
\frac{1}{(\mu \times \mu)(\varphi(x,y) \leq n)} \sum_{k=1}^n (\widehat{T \times T})^k \one_{(A\times A)^c \cap \{ \varphi(x,y)=k \}} \ \ \ \textrm{is uniformly bounded on } A \times A\,,
\end{equation}
where $\varphi(x,y) = \min \{n \geq 1: (T^nx, T^ny) \in A \times A \}$ is the first entrance time to $A \times A$, and $\widehat{T \times T}$ is a dual operator of $T \times T$. \\
(ii): $A$ is a uniformly returning set for $\one_A$ as specified in \eqref{e:uniformly_returning_set}. Moreover, $f$ is bounded.

Then, (\ref{e:main_result}) and (\ref{e:sub_main_result}) follow again.
\end{theorem}

Before stating the proof of Theorem \ref{t:main_theorem1}, we present three examples of different situations where the theorem applies. The first example is what \cite{resnick:samorodnitsky:xue:2000} studied, but their example can be regarded as a special case of our more general setup. 

\begin{example} \label{ex:null_rec_MC}
Let $( x_k, k \geq 0)$ be an irreducible null recurrent Markov chain with state space $\mathbb{Z}$ and transition matrix $P = ( p_{ij} )$. Let $P_i(\cdot)$ be a probability law of $( x_k )$ starting in state $i \in \mathbb{Z}$. Since $( x_k )$ is null recurrent, there exists a unique (up to constant multiplications), infinite, invariant measure $( \pi_i )$. We set $\pi_0=1$ for normalization. Define a $\sigma$-finite and infinite measure on $(E,\mathcal{E})=(\mathbb{Z}^{\bbn},\mathcal{B(\mathbb{Z}^{\bbn})})$ by 
$$
\mu(\cdot) = \sum_{i \in \mathbb{Z}} \pi_i P_i(\cdot)\,.
$$
Let $T:\mathbb{Z}^\bbn \to \mathbb{Z}^\bbn$ be the left shift map defined by $T(x_0, x_1, \dots) = (x_1, x_2, \dots)$. Obviously, $T$ preserves the measure $\mu$. From \cite{harris:robbins:1953}, it is known that the map $T$ is conservative and ergodic. 

We consider the set $A = \{ x \in \mathbb{Z}^{\mathbb{N}}: x_0 = 0 \}$. From the formula on page $157$ of \cite{aaronson:1997}, we have
$$
\widehat{T}^k \one_A (x) = P_0(x_k=0) \ \ \ \text{for } x \in A\,. 
$$
Thus, with the normalizing sequence $a_n =  \sum_{k=1}^n P_0(x_k=0)$, 
$$
\frac{1}{a_n} \sum_{k=1}^n \widehat{T}^k \one_A(x) = 1 = \mu(A)
$$
holds for every $x \in A$. Here, $A$ is a Darling-Kac set and hence $T$ is a pointwise dual ergodic map.

One of the possible ways for ensuring regular variation of $(a_n)$ is to assume 
$$
\sum_{k=1}^n P_0(\varphi \geq k) \in RV_{1-\beta} \ \ \ \text{for some } 0 \leq \beta < 1\,,
$$
where $\varphi(x) = \min\{n \geq 1: x_n = 0  \}$, $x \in \mathbb{Z}^{\bbn}$, is the first entrance time to the set $A$. \\
From Lemma 3.3 in \cite{resnick:samorodnitsky:xue:2000}, we see that $\mu(\varphi=n) = P_0(\varphi \geq n)$ and by \eqref{e:prop3.8.7},
$$
a_n \sim \frac{1}{\Gamma(2-\beta)\Gamma(1+\beta)} \frac{n}{\mu(\varphi \leq n)} \in RV_{\beta}\,.
$$

We will proceed to check condition \eqref{e:generalize_Markov_shift}. The formula on page $156$ of \cite{aaronson:1997} gives
$$
\widehat{T}^k \one_{A^c \cap \{ \varphi = k \}}\bigl( x_0,x_1,\ldots\bigr)
= \one_{\{x_0=0\}}\sum_{i_0 \not=0}\pi_{i_0}\sum_{i_1\not=0}p_{i_0 i_1}\ldots
\sum_{i_{k-1}\not=0}p_{i_{k-2} i_{k-1}}p_{i_{k-1} 0}\,,
$$
which immediately implies \eqref{e:generalize_Markov_shift}. 

We take a measurable function $f:\bbz^{\bbn} \to \bbr$ that is supported by the set $A$ and satisfies \eqref{e:integrability}. Now, Theorem \ref{t:main_theorem1} applies if the parameters lie in the range $1<\alpha<2$, $0 \leq \beta < 1$, or $0 < \alpha \leq 1$, $0 \leq \beta < 1/(2-\alpha)$. 
 
On the other hand, if $0 < \alpha \leq 1$ and $1/(2-\alpha) \leq \beta < 1$, we need to check the conditions given in $(i)$ of Theorem \ref{t:main_theorem1}. For this, we consider a two-dimensional Markov chain $(( x_k, y_k), \, k \geq0)$ with $(y_k)$ an independent copy of $(x_k)$. Let $P_{(i,j)}(\cdot)$ be a probability law of $(x_k, y_k)$ starting from $(i,j) \in \mathbb{Z} \times \mathbb{Z}$. It is now easy to check that $(x_k, y_k)$ is also irreducible and null recurrent, and a probability measure $\mu \times \mu$ can be written as 
$$
(\mu \times \mu)(\cdot) = \sum_{i \in \mathbb{Z}} \sum_{j \in \mathbb{Z}} \pi_i \pi_j P_{(i,j)}(\cdot)\,.
$$
Because of \cite{harris:robbins:1953} again, we can say that $T \times T$ is conservative ergodic and measure preserving map on $(\mathbb{Z}^\bbn \times \mathbb{Z}^\bbn, \mathcal{B}(\mathbb{Z}^\bbn) \times \mathcal{B}(\mathbb{Z}^\bbn))$. 

Evidently, the product set $A \times A$ is a Darling-Kac set. Indeed,
$$
\sum_{k=1}^n (\widehat{T \times T})^k \one_{A \times A}(x,y) = \sum_{k=1}^n \widehat{T}^k \one_A (x) \widehat{T}^k \one_A (y) 
= \sum_{k=1}^n P_0(x_k=0)^2 \ \ \ \text{for } (x,y) \in A \times A\,.
$$
Therefore, by the normalizing sequence $a_n^{\prime} = \sum_{k=1}^n P_0(x_k=0)^2$, the product set $A \times A$ turns out to be a Darling-Kac set, and $T \times T$ is, of course, pointwise dual ergodic. 

Once again, by appealing to Lemma 3.3 in \cite{resnick:samorodnitsky:xue:2000}, we get 
$$
(\mu \times \mu)(\varphi(x,y) \leq n) = \sum_{k=1}^n P_{(0,0)}(\varphi (x,y) \geq k) \in RV_{2(1-\beta)}.
$$
Thus,
$$
a_n^{\prime} \sim \frac{1}{\Gamma(3-2\beta) \Gamma(2\beta)} \frac{n}{(\mu \times \mu)(\varphi(x,y) \leq n)} \in RV_{2\beta-1}.
$$
To check \eqref{e:2dim_generalize_Markov_shift}, one more application of the formula on p. 156 of \cite{aaronson:1997} yields
$$
(\widehat{T \times T})^k \one_{(A \times A)^c \cap \{ \varphi(x,y) = k \}}\bigl( (x_0,y_0), (x_1,y_1) \ldots\bigr)
$$
$$
= \one_{\{(x_0,y_0)=(0,0)\}} \hspace{-5pt} \sum_{(i_0,j_0) \not=(0,0)} \pi_{i_0}\pi_{j_0} \hspace{-5pt} \sum_{(i_1,j_1)\not=(0,0)}p_{i_0 i_1}p_{j_0 j_1}\ldots
\hspace{-15pt} \sum_{(i_{k-1},j_{k-1})\not=(0,0)}\hspace{-15pt}p_{i_{k-2} i_{k-1}}p_{i_{k-1} 0}p_{j_{k-2} j_{k-1}}p_{j_{k-1} 0}\,.
$$
Therefore \eqref{e:2dim_generalize_Markov_shift} holds, and in this case, Theorem \ref{t:main_theorem1} applies as well. 

It is not difficult to prove that the process $\BX = (X_1,X_2,\dots)$ is mixing. To see this, we only check a sufficient condition proposed by Theorem 5 in \cite{rosinski:zak:1996}:
$$
\mu \{ x:|f(x)| > \epsilon, \, |f \circ T^n(x)| > \epsilon  \} \to 0 \ \ \ \text{as } n \to \infty, \ \ \text{for every } \epsilon >0\,.
$$
Since $f$ vanishes outside of $A$ and $(x_n)$ is null recurrent, we have as $n \to \infty$, 
$$
\mu \{ x:|f(x)| > \epsilon, \, |f \circ T^n(x)| > \epsilon  \} \leq \mu(A \cap T^{-n}A) = P_0(x_n=0) \to 0. 
$$
\end{example}
\vspace{10pt}

The next two examples are less familiar to probabilists, but are well known to ergodic theorists. 

\begin{example} \label{ex:basic_AFN_system}
In this example, we will define the so-called {\it basic AFN-system}. We refer the reader to \cite{zweimuller:2000} and to \cite{thaler:zweimuller:2006}. Let $E$ be the union of a finite family of disjoint bounded open intervals on $\bbr$, and let $\mathcal{E}$ be its Borel $\sigma$-field. Let $\xi$ be a (possibly infinite) collection of nonempty, pairwise disjoint open subintervals in $E$. With $\lambda$ being the one-dimensional Lebesgue measure, we assume $\lambda(E \setminus \bigcup_{Z \in \xi} Z)=0$. 

Let $T:E \to E$ be a twice-differentiable map and strictly monotonic on each $Z \in \xi$. Suppose that $T$ satisfies the following conditions. \\
$(A)$ {\it Adler's condition}:
$$
T^{\prime \prime} / (T^{\prime})^2 \text{ is bounded on } \bigcup_{Z \in \xi} Z\,.
$$
$(F)$ {\it Finite image condition}:
$$
\text{the collection } T \xi = \{ TZ: Z \in \xi  \} \text{ is finite}.
$$
$(N)$ {\it A possibility of nonuniform expansion}: There exists a finite subset $\zeta \subseteq \xi$ such that each $Z \in \zeta$ has {\it an indifferent fixed point} $x_Z$ as one of its endpoints. That is, 
$$
\lim_{x \to x_Z, x \in Z} Tx = x_Z \ \ \ \text{and} \ \ \ \lim_{x \to x_Z, x \in Z} T^{\prime} x = 1\,.
$$
Moreover, we suppose that for each $Z \in \zeta$, 
$$
T^{\prime} \text{ either decreases on } (-\infty, x_Z) \cap Z \text{ or increases on } (x_Z, \infty) \cap Z\,,
$$ 
depending on whether $x_Z$ is the left endpoint or the right endpoint of $Z$. 

Assume that $T$ is uniformly expanding whenever it is bounded away from the family of indifferent fixed points $\{x_Z : Z \in \zeta  \}$, i.e., for each $\epsilon > 0$, there is a $\rho(\epsilon) > 1$ such that 
$$
|T^{\prime}| \geq \rho(\epsilon) \text{ on } E \setminus \bigcup_{Z \in \zeta} \bigl((x_Z - \epsilon, x_Z + \epsilon) \cap Z \bigr)\,.
$$

We will further specify the behavior of $T$ in neighborhood of the indifferent fixed points; for every $Z \in \zeta$, there is $0 \leq \beta_{Z} < 1$ such that 
\begin{equation}  \label{e:nice_expansion}
Tx = x + a_Z |x-x_Z|^{1/\beta_Z + 1} + o(|x-x_Z|^{1/\beta_Z + 1}) \ \ \ \text{as } x \to x_Z \text{ in } Z
\end{equation}
for some $a_Z \neq 0$. 

As argued in \cite{zweimuller:2000}, there is not much loss of generality in assuming that $T$ is conservative and ergodic with respect to $\lambda$. In this case, if additionally $\zeta$ is non-empty, then the triplet $(E,T,\xi)$ is said to be {\it a basic AFN-system}. In the sequel, we will assume this property. 

Given a basic AFN-system $(E,T,\xi)$, there always exists an infinite invariant measure $\mu \ll \lambda$ with density $d\mu / d\lambda(x) = h_0(x)G(x)$, where 
$$
G(x) = \begin{cases} (x-x_Z)\bigl(x-(T|_Z)^{-1}(x)\bigr)^{-1} & \text{if } x \in Z \in \zeta\,, \\
  1  & \text{if } x \in E \setminus \bigcup_{Z \in \zeta}Z\,,
\end{cases}
$$
and $h_0$ is a function of bounded variation, bounded away from both zero and infinity. Now, we can view $T$ as a conservative ergodic and measure preserving map on an infinite measure space $(E,\mathcal{E},\mu)$.

An example of a basic AFN-map is Boole's transformation placed on
$E=(0,1/2)\cup (1/2,1)$, defined by
$$
T(x) = \frac{x(1-x)}{1-x-x^2} , \ x\in (0,1/2), \ \ T(x)=1-T(1-x), \
x\in (1/2,1)\,.
$$
It admits expansions of the form \eqref{e:nice_expansion} at the indifferent fixed points $x_Z=0$ and
$x_Z=1$ with $\beta_Z=1/2$ in both cases. The invariant measure $\mu$
satisfies 
$$
\frac{d\mu}{d\lambda}(x) = \frac{1}{x^2}+ \frac{1}{(1-x)^2}, \ x\in
E\,; 
$$
see page 4-5 in \cite{thaler:2001}.

Given a constant $0 < \epsilon <1$, we take
$$
A = E \setminus \bigcup_{Z \in \zeta} \bigl((x_Z - \epsilon, x_Z + \epsilon) \cap Z \bigr)\,.
$$ 

Since $\lambda(\partial A)=0$ and $A$ is bounded away from $\{ x_Z, Z \in \zeta \}$, $A$ is a Darling-Kac set, and hence $T$ is a pointwise dual ergodic map (Corollary 3 in \cite{zweimuller:2000}). Furthermore, because of the assumption \eqref{e:nice_expansion}, $(a_n)$ turns out to be regularly varying with index $\beta=\min_{Z \in \zeta}\beta_Z$ (Theorem 4 in \cite{zweimuller:2000}). Moreover, the formulas (2.5) and (2.6) in \cite{thaler:zweimuller:2006} prove the condition \eqref{e:generalize_Markov_shift}. 

Suppose that the parameters $\alpha$ and $\beta$ lie in the range of either $1<\alpha<2$, $0\leq \beta < 1$, or $0 < \alpha \leq 1$, $0 \leq \beta < 1/(2-\alpha)$. If a measurable function $f:E \to \bbr$ is supported by the set $A$ together with a proper integrability assumption, then Theorem \ref{t:main_theorem1} applies. 

Suppose that $0 < \alpha \leq 1$ and $1/(2-\alpha) \leq \beta < 1$. In this case, we will check $(ii)$ in Theorem \ref{t:main_theorem1}, because unlike Example \ref{ex:null_rec_MC}, the product map $T \times T$ is not generally conservative and ergodic.  According to condition $(ii)$, however, the Darling-Kac set $A$ must be a uniformly returning set. Unfortunately, this is not always the case for a general basic AFN-system. 
To overcome this difficulty, we have to impose certain additional assumptions; see for example, \cite{thaler:2000}. If we restrict ourselves to such a type of a basic AFN-system, then $(ii)$ is satisfied and consequently Theorem \ref{t:main_theorem1} follows.

Finally, it is worth pointing out that the process $\BX=(X_1,X_2,\dots)$ is mixing. This can be proved as in Example \ref{ex:null_rec_MC}. 
\end{example}

\begin{example} \label{ex:S_unimodal}
We will construct the dynamical system by a $S$-unimodal map with flat critical point. The main reference here is \cite{zweimuller:2004}. Let $T:[a,b] \to [a,b]$ be a $S$-unimodal map with flat critical point $c \in (a,b)$. That is, the Schwarzian derivative of $T$ is nonpositive: $\mathcal{S}T = T^{\prime \prime \prime} / T^{\prime} - \frac{3}{2} (T^{\prime \prime} / T^{\prime})^2 \leq 0$, and all derivatives at the critical point $c$ vanish: $T^{(n)}c = 0$ for all $n \geq 1$. Further assume that $Ta = Tb = a$ and that $\int_{[a,b]} \ln |T^{\prime}| d\lambda = -\infty$ ($\lambda$ is the one-dimensional Lebesgue measure). In addition, we suppose that $T$ satisfies Misiurewicz condition, i.e., there is an open interval $I$ containing $c$ such that $T^n c \notin I$ for all $n \geq 1$. Also, assume that there exists a positive and finite Lyapunov exponent $\lambda_c = \lim_{n \to \infty} n^{-1} \ln |(T^n)^{\prime}(Tc)|$.

The dynamical effect of a flat critical point is that the closer the orbit gets to $c$, the slower it moves away from the critical orbit $( T^n c, n \geq 1 )$. Consequently, the orbit stays in neighborhood of $( T^n c, n \geq 1 )$ for a nonnegligible amount of time. 

It is shown in \cite{zweimuller:2004} that there exists an infinite measure $\mu \ll \lambda$ such that $T$ is a conservative ergodic and measure preserving map on $([a,b],\mathcal{B}([a,b]),\mu)$. 

From \cite{zweimuller:2004} and \cite{zweimuller:2007}, one can find a Darling-Kac set $A$, which is bounded away from the critical orbit $( T^n c, n \geq 1 )$ such that 
$$
\left( \frac{\widehat{T}^n \one_{A \cap \{ \varphi = n \}}}{\mu(A \cap \{ \varphi = n \})}, n \geq 1 \right) \ \ \ \text{is bounded on } A. 
$$
This property in fact proves condition \eqref{e:generalize_Markov_shift}. The existence of a positive and finite Lyapunov exponent guarantees that the normalizing sequence $(a_n)$ for the Darling-Kac set is a regularly varying function of the order $0 < \beta < 1$ (Theorem 7 in \cite{zweimuller:2004}). 

Suppose that the range of the parameters $\alpha$ and $\beta$ is either $1<\alpha<2$, $0 < \beta < 1$ or $0 < \alpha \leq 1$, $0 < \beta < 1/(2-\alpha)$. If a measurable function $f$ satisfies a proper integrability condition and is supported by the set $A$, Theorem \ref{t:main_theorem1} applies. 
\end{example}

\section{Important Ingredients} \label{sec:lemmas}

For the completion of the proof of Theorem \ref{t:main_theorem1}, we need several important ingredients, all of which are collected in this section. The first result is the most important and is known as ''Generalized Darling-Kac theorem'', which describes ergodic convergence of partial sums when the trajectory $(T^n x)$ is depicted by a pointwise dual ergodic map. This is of interest on its own in infinite ergodic theory; see \cite{aaronson:1981}, \cite{thaler:zweimuller:2006}, and \cite{owada:samorodnitsky:2012}. The first part of the proof of Theorem \ref{t:main_theorem1} is closely related to the argument in \cite{thaler:zweimuller:2006}. In particular, as seen in \cite{thaler:zweimuller:2006}, the idea of constant applications of Karamata's Tauberian theorem for power series (e.g., Corollary 1.7.3 in \cite{bingham:goldie:teugels:1987} or Proposition 4.2 in \cite{thaler:zweimuller:2006}) is crucial. The useful techniques to handle the power series below are collected in Section 5 of \cite{thaler:zweimuller:2006}. 

\begin{lemma} (Generalized Darling-Kac theorem) \label{l:DK_theorem}
Under the assumptions of Theorem \ref{t:main_theorem1}, let $\phi(x) = f(x) \sum_{h=0}^H \theta_h f_h(x)$, $\theta_0, \dots, \theta_H \in \bbr$. Then, we have as $n \to \infty$,
$$
\frac{S_n(\phi)}{a_n} \Rightarrow \mu(\phi) \Gamma(1+\beta) M_{\beta}(1-V_{\beta}) \ \ \ \text{in } \mathbb{R}
$$
with respect to $\mu_n(\cdot) = \mu(\cdot \cap \{ \varphi \leq n \}) / \mu(\varphi \leq n)$. Here, $(M_{\beta}(t), \, t \geq 0)$ and $V_{\beta}$ are defined in \eqref{e:MLprocess} and \eqref{e:pdfTinf}, respectively. 
\end{lemma}

\begin{proof}
We first claim that as $n \to \infty$,
\begin{equation}  \label{e:firstgoalergodic}
\frac{S_n(\one_A)}{a_n} \Rightarrow \mu(A) \Gamma(1+\beta) M_{\beta} (1-V_{\beta}) \ \ \ \text{in } \mathbb{R} 
\end{equation}
with respect to $\mu_n$, and we will try to replace $\one_A$ by a more general function $\phi$ thereafter. 
Because of \eqref{e:mgf_Mittag_Leffler} and the fact that $M_{\beta}(t)$ is a self-similar process with self-similarity exponent $\beta$, the moments of $M_{\beta}(1-V_{\beta})$ are given by
$$
E M_{\beta}(1-V_{\beta})^r = E (1-V_{\beta})^{r \beta} E M_{\beta}(1)^r = r! \frac{\Gamma(2-\beta)}{\Gamma(r \beta + 2 - \beta)}\,.
$$
Recall the fact that given the moments of all orders, the Mittag-Leffler laws can be uniquely determined (e.g., \cite{bingham:1971}). A simple application of the Carleman sufficient condition proves that the laws of $M_{\beta}(1-V_{\beta})$ can also be uniquely determined by their moments. From these observations, (\ref{e:firstgoalergodic}) follows if we can show that as $n \to \infty$,
$$
\int_E \left( \frac{S_n(\one_A)}{a_n} \right)^r d\mu_n \to \bigl( \mu(A) \Gamma(1+\beta) \bigr)^r r! \frac{\Gamma(2-\beta) }{\Gamma(r \beta + 2 - \beta)}\,, \ \ \ \text{for every } r=1,2,\dots
$$

First, we claim that 
\begin{equation}  \label{e:intermediate_KTT}
\sum_{n=1}^{\infty} \left( \int_E \begin{pmatrix} S_n(\one_A) \\ r \end{pmatrix} d\mu \right) e^{-\lambda n} \sim \frac{1}{(r-1)!} \frac{\mu(A)}{\lambda} \sum_{n=1}^{\infty} \left( \int_A S_n(\one_A)^{r-1} d\mu_A \right) e^{-\lambda n} \ \ \ \text{as } \lambda \downarrow 0\,, 
\end{equation}
where $\mu_A(\cdot) = \mu(\cdot \cap A) / \mu(A)$. \\
For the proof, the following identity is needed:
$$
\begin{pmatrix} S_n(\one_A) \\ r \end{pmatrix} = \sum_{k=1}^n \left( \one_A \begin{pmatrix} S_{n-k}(\one_A) \\ r-1 \end{pmatrix} \right) \circ T^k, \ \ \ r=1,2,\dots.
$$
As $\lambda \downarrow 0$, we have
$$
\sum_{n=1}^{\infty} \left( \int_E \begin{pmatrix} S_n(\one_A) \\ r \end{pmatrix} d\mu \right) e^{-\lambda n} = \sum_{n=1}^{\infty} \sum_{k=1}^n \left( \int_E \left( \one_A \begin{pmatrix} S_{n-k}(\one_A) \\ r-1 \end{pmatrix} \right) \circ T^k d\mu \right) e^{-\lambda n} 
$$
$$
\sim \frac{\mu(A)}{\lambda} \sum_{n=1}^{\infty} \left( \int_A \begin{pmatrix} S_n(\one_A) \\ r-1 \end{pmatrix} d\mu_A \right) e^{-\lambda n}.
$$
It is elementary to show that
$$
\int_A \begin{pmatrix} S_n(\one_A) \\ r-1 \end{pmatrix} d\mu_A \sim \frac{1}{(r-1)!} \int_A S_n(\one_A)^{r-1} d\mu_A \ \ \ \text{as } n \to \infty\,,
$$
which completes \eqref{e:intermediate_KTT}. 

We already know from the proof of Theorem 9.1 in \cite{thaler:zweimuller:2006} (or \cite{aaronson:1981}) that
$$ 
\int_A S_n(\one_A)^{r-1} d\mu_A \sim \bigl( \mu(A) \Gamma(1+\beta) \bigr)^{r-1} E M_{\beta}(1)^{r-1} a_n^{r-1} 
$$
$$
= \bigl( \mu(A) \Gamma(1+\beta) \bigr)^{r-1} (r-1)! \frac{a_n^{r-1}}{\Gamma((r-1)\beta + 1)}   \ \ \ \text{as } n \to \infty\,.
$$
Since $(a_n)$ is regularly varying with exponent $\beta$, one can set $a_n = n^{\beta} L(n)$ by some slowly varying function $L$. Then, from Karamata's Tauberian theorem, 
\begin{equation}  \label{e:Theorem9.1}
\sum_{n=1}^{\infty} \left( \int_A S_n(\one_A)^{r-1} d\mu_A \right) e^{-\lambda n} \sim (r-1)! (\mu(A) \Gamma(1+\beta))^{r-1} \frac{1}{\lambda^{(r-1)\beta + 1}} L(\lambda^{-1})^{r-1} \ \ \ \text{as } \lambda \downarrow 0\,. 
\end{equation}
Consequently, from \eqref{e:intermediate_KTT} and \eqref{e:Theorem9.1}, 
$$
\sum_{n=1}^{\infty} \left( \int_E \begin{pmatrix} S_n(\one_A) \\ r \end{pmatrix} d\mu \right) e^{-\lambda n} \sim \mu(A)^r \Gamma(1+\beta)^{r-1} \frac{1}{\lambda^{r\beta+2-\beta}} L(\lambda^{-1})^{r-1} \ \ \ \text{as } \lambda \downarrow 0\,. 
$$
Since $\int_E \begin{pmatrix} S_n(\one_A) \\ r \end{pmatrix} d\mu$ is nondecreasing in $n$ and $r\beta+2-\beta>0$, one more application of Karamata's Tauberian theorem yields
$$
\int_E \begin{pmatrix} S_n(\one_A) \\ r \end{pmatrix} d\mu \sim \frac{\mu(A)^r \Gamma(1+\beta)^{r-1}}{\Gamma(r \beta + 2 - \beta)} n a_n^{r-1} \ \ \ \text{as } n \to \infty\,.
$$
It is not difficult to justify 
$$
\int_E \begin{pmatrix} S_n(\one_A) \\ r \end{pmatrix} d\mu \sim \frac{1}{r!} \int_E S_n(\one_A)^r d\mu \ \ \ \text{as } n \to \infty\,.
$$
Therefore, we get
$$
\int_E \left( \frac{S_n(\one_A)}{a_n} \right)^r d\mu \sim \mu(A)^r r! \frac{\Gamma(1+\beta)^{r-1}}{\Gamma(r\beta+2-\beta)} \frac{n}{a_n} \ \ \ \text{as } n \to \infty\,. 
$$
Thus we get, from \eqref{e:equivalent_wandering} and \eqref{e:prop3.8.7}, 
$$
\int_E \left( \frac{S_n(\one_A)}{a_n} \right)^r d\mu_n = \frac{1}{\mu(\varphi \leq n)} \int_E \left( \frac{S_n(\one_A)}{a_n} \right)^r d\mu 
$$
$$
\to \mu(A)^r r! \frac{\Gamma(2-\beta) \Gamma(1+\beta)^r}{\Gamma(r\beta+2-\beta)} \ \ \ \text{as } n \to \infty\,,
$$
which completes \eqref{e:firstgoalergodic}. 

Next, the indicator function $\one_A$ must be replaced by $\phi$. To this end, it suffices to show that
\begin{equation}  \label{e:goal_replacement}
\mu_n \left( \left| \frac{S_n(\phi)}{S_n(\one_A)} - \frac{\mu(\phi)}{\mu(A)} \right| > \epsilon \right) \to 0 \ \ \ \text{as } n \to \infty, \ \text{for every } \epsilon > 0\,.
\end{equation}
Indeed, if \eqref{e:goal_replacement} is true, the Slutsky theorem gives
$$
\left( \frac{S_n(\one_A)}{a_n}, \frac{S_n(\phi)}{S_n(\one_A)} \right) \Rightarrow \left( \mu(A) \Gamma(1+\beta) M_{\beta}(1-V_{\beta}), \frac{\mu(\phi)}{\mu(A)} \right)
$$
with respect to $\mu_n$. Applying the continuous mapping theorem, we get as $n \to \infty$,
$$
\frac{S_n(\phi)}{a_n} \Rightarrow \mu(\phi) \Gamma(1+\beta) M_{\beta}(1-V_{\beta}) \ \ \ \text{in } \mathbb{R}. \label{eq:extendtophi}
$$ 
Since $\mu(A)<\infty$, it is now enough to verify 
$$
\mu_n \left(A^c \cap \left\{ \left| \frac{S_n(\phi)}{S_n(\one_A)} - \frac{\mu(\phi)}{\mu(A)} \right| > \epsilon \right\} \right) \to 0 \ \ \ \text{as } n \to \infty, \ \text{for every } \epsilon > 0\,.
$$
Denote 
$$
K_n = \left\{ \left| \frac{\phi + S_n(\phi)}{1 + S_n(\one_A)} - \frac{\mu(\phi)}{\mu(A)} \right| > \epsilon  \right\}\,.
$$
Noting that $\phi$ is supported by $A$, we obtain 
$$
\mu \left( A^c \cap \{ \varphi \leq n \} \cap \left\{ \left| \frac{S_n(\phi)}{S_n(\one_A)} - \frac{\mu(\phi)}{\mu(A)} \right| > \epsilon \right\} \right) = \sum_{m=1}^n \mu \bigl(A^c \cap \{ \varphi = m \} \cap T^{-m} K_{n-m} \} \bigr)\,.
$$
Thus, for an arbitrary constant $\delta \in (0,1)$, one can proceed as follows.
$$
\mu_n \left(A^c \cap \left\{ \left| \frac{S_n(\phi)}{S_n(\one_A)} - \frac{\mu(\phi)}{\mu(A)} \right| > \epsilon \right\} \right)
$$
$$
\leq \frac{1}{\mu(\varphi \leq n)} \sum_{m=1}^{\lceil (1-\delta)n \rceil} \mu\bigl(A^c \cap \{ \varphi = m \} \cap T^{-m} K_{n-m}\bigr) + \frac{1}{\mu(\varphi \leq n)}  \sum_{m= \lceil (1-\delta) n \rceil + 1}^n \mu(\varphi = m) 
$$
$$
=\int_A \frac{1}{\mu(\varphi \leq n)} \sum_{m=1}^{\lceil (1-\delta)n \rceil} \widehat{T}^m \one_{A^c \cap \{ \varphi = m \}} \cdot \one_{K_{n-m}} d\mu + 1 - \frac{\mu(\varphi \leq \lceil (1-\delta)n \rceil)}{\mu(\varphi \leq n)} 
$$
$$
\leq \int_A \frac{1}{\mu(\varphi \leq \lceil (1-\delta) n \rceil)} \sum_{m=1}^{\lceil (1-\delta)n \rceil} \widehat{T}^m \one_{A^c \cap \{ \varphi = m \}} \hspace{-10pt} \sup_{n- \lceil (1-\delta)n \rceil \leq i \leq n} \hspace{-10pt} \one_{K_i} d\mu + 1 - \frac{\mu(\varphi \leq \lceil (1-\delta)n \rceil)}{\mu(\varphi \leq n)}\,.
$$
Because of \eqref{e:generalize_Markov_shift}, $\mu(\varphi \leq \lceil (1-\delta) n \rceil)^{-1} \sum_{m=1}^{\lceil (1-\delta)n \rceil} \widehat{T}^m \one_{A^c \cap \{ \varphi = m \}}$ is uniformly bounded on $A$; further, the Hopf's ergodic theorem (sometimes also called a ratio ergodic theorem; see Theorem 2.2.5 in \cite{aaronson:1997}) yields  
$$
\sup_{n - \lceil (1-\delta)n \rceil \leq i \leq n} \hspace{-20pt} \one_{K_i}  \to 0 \ \ \ \text{as } n \to \infty \ \ \ \text{a.e. on } A\,. 
$$
Applying the dominated convergence theorem, we conclude
$$
\limsup_{n \to \infty} \mu_n \left(A^c \cap \left\{ \left| \frac{S_n(\phi)}{S_n(\one_A)} - \frac{\mu(\phi)}{\mu(A)} \right| > \epsilon \right\} \right) \leq 1 - (1-\delta)^{1-\beta}\,.
$$
Letting $\delta \downarrow 0$ on the right hand side, we get \eqref{e:goal_replacement}. 
\end{proof}

Another important ingredient that plays a crucial role in the proof of Theorem \ref{t:main_theorem1} is  
\begin{equation}  \label{e:imp.ingredient}
E \left| \frac{1}{c_n} \sum_{k=1}^n f_k(V_i) f_{k+h}(V_j) q(V_i)^{-1/\alpha^{\prime}} q(V_j)^{-1/\alpha^{\prime}} \right|^{\alpha^{\prime}} \to 0, \ \ \ \text{as } n \to \infty, \ \textrm{for } i \neq j\,.
\end{equation}
where the random variables $(V_i)$ are defined in \eqref{e:series.representation}, and $\alpha^{\prime} = \alpha - \xi$ by some positive constant $\xi>0$. This result will be proved in Lemmas \ref{l:L_alpha_conv_1} - \ref{l:L_alpha_conv_3} below. The constant $\xi>0$ varies in accordance with the values of $\alpha$ and $\beta$. Lemma \ref{l:L_alpha_conv_1} treats the case when $(\alpha,\beta)$ lies in the range (\ref{e:first_range}), while Lemmas \ref{l:L_alpha_conv_2} and \ref{l:L_alpha_conv_3} apply when $(\alpha,\beta)$ lies outside of the range (\ref{e:first_range}). 

\begin{lemma} \label{l:L_alpha_conv_1}
Let $\alpha$ and $\beta$ be in the range of \eqref{e:first_range}. Fix a constant $\xi > 0$ such that
\begin{align*}
\xi &< \alpha - 1 \ \ \ \text{if } 1< \alpha < 2\,,  \\ 
\xi &< \alpha \left( 1 - \frac{1}{2-\beta(2-\alpha)} \right) \ \ \ \text{if } 0 < \alpha \leq 1, \; 0 \leq \beta < \frac{1}{2-\alpha}\,.
\end{align*}
Let $\alpha^{\prime} = \alpha - \xi$. Then, under the conditions of Theorem \ref{t:main_theorem1}, \eqref{e:imp.ingredient} holds.
\end{lemma}

\begin{proof} 
First, suppose that $1 < \alpha < 2$. Since $\alpha^{\prime} > 1$, Minkowski's inequality applies to obtain
\begin{align*}
E \left| \frac{1}{c_n} \sum_{k=1}^n f_k(V_i) f_{k+h}(V_j) q(V_i)^{-1/\alpha^{\prime}} q(V_j)^{-1/\alpha^{\prime}} \right|^{\alpha^{\prime}} 
&= \frac{1}{c_n^{\alpha^{\prime}}} \int_{E \times E} \bigl|\sum_{k=1}^n f_k(x) f_k(y) \bigr|^{\alpha^{\prime}} (\mu \times \mu) (dx \; dy) \\
&\leq \left( \frac{n}{c_n} \right)^{\alpha^{\prime}} \left( \int_E |f|^{\alpha^{\prime}} d\mu \right)^2\,.
\end{align*}
Since $n/c_n \in RV_{(1-\beta)(1-2/\alpha)}$ with $(1-\beta)(1-2/\alpha) < 0$, we have $n/c_n \to 0$.  

Next, suppose that $0 < \alpha \leq 1$ and $0 \leq \beta < 1/(2-\alpha)$. In this case, a simple application of the triangle inequality gives 
\begin{align*}
E \left| \frac{1}{c_n} \sum_{k=1}^n f_k(V_i) f_{k+h}(V_j) q(V_i)^{-1/\alpha^{\prime}} q(V_j)^{-1/\alpha^{\prime}} \right|^{\alpha^{\prime}} 
&= \frac{1}{c_n^{\alpha^{\prime}}} \int_{E \times E} \bigl|\sum_{k=1}^n f_k(x) f_k(y) \bigr|^{\alpha^{\prime}} (\mu \times \mu) (dx \; dy) \\
&\leq  \frac{n}{c_n^{\alpha^{\prime}}} \left( \int_E |f|^{\alpha^{\prime}} d\mu \right)^2\,.
\end{align*}
Since $n / c_n^{\alpha^{\prime}} \in RV_{1-\alpha^{\prime} (\beta + 2(1-\beta) / \alpha)}$ with $1-\alpha^{\prime} (\beta + 2(1-\beta) / \alpha) < 0$, we have $n/c_n^{\alpha^{\prime}} \to 0$. 
\end{proof}

\begin{lemma} \label{l:L_alpha_conv_2}
Under the setup of Theorem \ref{t:main_theorem1}, particularly let $0 < \alpha \leq 1$, $1/(2-\alpha) \leq \beta < 1$ and assume condition $(i)$. Fix $0 < \xi < \alpha^2 / (\alpha+2)$ and let $\alpha^{\prime} = \alpha - \xi$. Then, (\ref{e:imp.ingredient}) follows. 
\end{lemma}

\begin{proof}
Denote by $S_n(f \times f)(x,y) = \sum_{k=1}^n f_k(x) f_k(y)$ a partial sum defined on a product space $E \times E$. 
By virtue of \eqref{e:2dim_generalize_Markov_shift}, proceeding as in the proof of Lemma \ref{l:DK_theorem}, we can get as $n \to \infty$,
$$
\frac{S_n(f \times f)(x,y)}{a_n^{\prime}} \Rightarrow \mu(f)^2 \Gamma(2\beta) M_{2\beta-1}(1-V_{2\beta-1}) \ \ \ \text{in } \mathbb{R}\,,
$$
where the weak convergence takes place under a probability measure 
\begin{equation}  \label{e:needconservativity}
(\mu \times \mu)_n (\cdot) = (\mu \times \mu)(\cdot \cap \{ \varphi(x,y) \leq n \}) / (\mu \times \mu) (\varphi(x,y) \leq n)\,. 
\end{equation} 
Here, $M_{2\beta-1}(t)$ is the Mittag-Leffler process with exponent $2\beta-1$, and $V_{2\beta-1}$ is defined by \eqref{e:pdfTinf}. The reader may, once again, refer to \cite{thaler:zweimuller:2006}.\\
From \eqref{e:equivalent_wandering} and \eqref{e:prop3.8.7}, and the assumption that $T \times T$ is a conservative and ergodic map, we can obtain
$$
(\mu \times \mu)(\varphi(x,y) \leq n) \sim \frac{1}{\Gamma(3-2\beta) \Gamma(2\beta)} \frac{n}{a_n^{\prime}} \ \ \ \text{as } n \to \infty\,,
$$
from which $(\mu \times \mu)(\varphi(x,y) \leq n) \in RV_{2(1-\beta)}$ follows. \\ 
Now, we have 
$$
E \Bigl| \sum_{k=1}^n f_k(V_i) f_{k+h}(V_j) q(V_i)^{-1/\alpha^{\prime}} q(V_j)^{-1/\alpha^{\prime}} \Bigr|^{\alpha^{\prime}} 
= \int_{E \times E} |S_n(f \times f)(x,y)|^{\alpha^{\prime}} (\mu \times \mu) (dx \; dy) 
$$
$$
= (a_n^{\prime})^{\alpha^{\prime}} (\mu \times \mu)(\varphi(x,y) \leq n) \int_{E \times E} \left| \frac{S_n(f \times f)(x,y)}{a_n^{\prime}} \right|^{\alpha^{\prime}} (\mu \times \mu)_n (dx \; dy)\,. 
$$
Note that
$$
\sup_{n \geq 1} \int_{E \times E} \left| \frac{S_n(f \times f)(x,y)}{a_n^{\prime}} \right| (\mu \times \mu)_n (dx \; dy) \leq \sup_{n \geq 1} \frac{n \mu(f)^2}{a_n^{\prime} (\mu \times \mu)(\varphi(x,y) \leq n)} < \infty\,.
$$
This means uniform integrability of $(|S_n(f \times f) / a_n^{\prime}|^{\alpha^{\prime}}, n\geq 1)$ with respect to $(\mu \times \mu)_n$, and hence, we have 
$$
\int_{E \times E} \left| \frac{S_n(f \times f)(x,y)}{a_n^{\prime}} \right|^{\alpha^{\prime}} (\mu \times \mu)_n (dx \; dy) \to \mu(f)^{2\alpha^{\prime}} \Gamma(2\beta)^{\alpha^{\prime}} E M_{2\beta-1}(1-V_{2\beta-1})^{\alpha^{\prime}} < \infty\,.
$$
On the other hand, \eqref{e:RV_index_cn} implies $c_n^{\alpha^{\prime}} \in RV_{\alpha^{\prime}(\beta + 2(1-\beta)/\alpha)}$. 
Thus, 
$$
E \left| \frac{1}{c_n} \sum_{k=1}^n f_k(V_i) f_{k+h}(V_j) q(V_i)^{-1/\alpha^{\prime}} q(V_j)^{-1/\alpha^{\prime}} \right|^{\alpha^{\prime}} \in RV_{(2\beta-1) \alpha^{\prime} + 2(1-\beta) - \alpha^{\prime}(\beta + 2(1-\beta)/\alpha))}.
$$
Owing to the constraint on $\xi$, we have $(2\beta-1)\alpha^{\prime} + 2(1-\beta) - \alpha^{\prime}(\beta + 2(1-\beta)/\alpha) < 0$, and hence, \eqref{e:imp.ingredient} is obtained.
\end{proof}

\begin{lemma}  \label{l:L_alpha_conv_3}
Under the setup of Theorem \ref{t:main_theorem1}, particularly let $0 < \alpha \leq 1$, $1/(2-\alpha) \leq \beta < 1$ and assume condition $(ii)$. Fix $0 < \xi < \alpha^2/(\alpha+2)$ and let $\alpha^{\prime} = \alpha - \xi$. Then, (\ref{e:imp.ingredient}) follows. 
\end{lemma}

\begin{proof}
We start by claiming that as $n \to \infty$,
\begin{equation}  \label{e:firststep_lemma3.3}
\frac{1}{a_n^{\prime}} \sum_{k=1}^n (\widehat{T \times T})^k \one_{A \times A}(x,y) \to \mu(A)^2 \ \ \ \text{uniformly, a.e. on } A \times A\,, 
\end{equation}
where
$$
a_n^{\prime} = \left( \frac{\Gamma(1+\beta)}{\Gamma(\beta)} \right)^2 \frac{\Gamma(2\beta-1)}{\Gamma(2\beta)} \frac{a_n^2}{n}\,.
$$
Indeed, from \eqref{e:bn_wn} and \eqref{e:uniformly_returning_set}, we see that as $n \to \infty$, 
$$
\sum_{k=1}^n (\widehat{T \times T})^k \one_{A \times A}(x,y) = \sum_{k=1}^n \widehat{T}^k \one_A(x) \widehat{T}^k \one_A(y) 
$$
$$
\sim \frac{\mu(A)^2}{\Gamma(\beta)^2 \Gamma(2-\beta)^2} \sum_{k=1}^n \frac{1}{w_k^2} \ \ \ \text{uniformly, a.e. on } A \times A\,.
$$
Applying Karamata's Tauberian theorem for power series to relation \eqref{e:prop3.8.7}, 
$$
\sum_{k=1}^n \frac{1}{w_k^2} \sim \Gamma(2-\beta)^2 \Gamma(1+\beta)^2 \frac{\Gamma(2\beta-1)}{\Gamma(2\beta)} \frac{a_n^2}{n} \ \ \ \text{as } n \to \infty\,.
$$
Thus, \eqref{e:firststep_lemma3.3} is obtained. 

Now, (\ref{e:firststep_lemma3.3}) ensures that $A \times A$ can be viewed as a Darling-Kac set for the product map $T \times T$. Thus, a careful inspection of Theorem 9.1 in \cite{thaler:zweimuller:2006} reveals that even if $T \times T$ is neither conservative nor ergodic, 
$$
\int_{E \times E} \left( \frac{S_n(\one_{A \times A})(x,y)}{a_n^{\prime}} \right)^r (\mu \times \mu) (dx \; dy)  \sim \mu(A)^{2r} r! \frac{\Gamma(2\beta)^{r-1}}{\Gamma(r(2\beta-1)+3-2\beta)} \frac{n}{a_n^{\prime}} \ \ \ \text{as } n \to \infty\,.
$$
Finally, we define a probability measure $(\mu \times \mu)_n(\cdot)$ by 
$$
(\mu \times \mu)_n (\cdot) = (\mu \times \mu)(\{ \varphi(x) \leq n, \varphi(y) \leq n \} \cap \cdot) / \mu(\varphi \leq n)^2
$$
(note that the above definition of $(\mu \times \mu)_n$ differs from (\ref{e:needconservativity})). Then, we have
$$
\int_{E \times E} \left( \frac{S_n(\one_{A \times A})(x,y)}{a_n^{\prime}} \right)^r (\mu \times \mu)_n (dx \; dy) = \frac{1}{\mu(\varphi \leq n)^2} \int_{E \times E} \left( \frac{S_n(\one_{A \times A})(x,y)}{a_n^{\prime}} \right)^r (\mu \times \mu) (dx \; dy)
$$
$$
\to \mu(A)^{2r} r! \frac{\Gamma(\beta)^2 \Gamma(2-\beta)^2 \Gamma(2\beta)^r}{\Gamma(2\beta-1) \Gamma(r(2\beta-1)+3-2\beta)}  \equiv \mu(A)^{2r} \eta_r \ \ \ \text{as } n \to \infty\,.
$$
The sequence $(\eta_r)$ determines, uniquely in law, a random variable $Z_{\beta}$, whose $r$th moment coincides with $\eta_r$ itself. To see this, it is enough to check the Carleman sufficient condition $\sum_{k=1}^{\infty} \eta_{2k}^{-1/2k} = \infty$; this can be easily checked by Stirling's formula together with elementary algebra. It is thus concluded that with respect to $(\mu \times \mu)_n$, as $n \to \infty$,
$$
\frac{S_n(\one_{A \times A})(x,y)}{a_n^{\prime}} \Rightarrow \mu(A)^2 Z_{\beta} \ \ \ \text{in } \mathbb{R}\,.
$$
Since $f$ is bounded and is supported by $A$, there is a constant $C_1 > 0$ such that
$$
E \Bigl| \sum_{k=1}^n f_k(V_i) f_{k+h}(V_j) q(V_i)^{-1/\alpha^{\prime}} q(V_j)^{-1/\alpha^{\prime}} \Bigr|^{\alpha^{\prime}} 
= \int_{E \times E} |S_n(f \times f)(x,y) |^{\alpha^{\prime}} (\mu \times \mu) (dx \; dy)
$$
$$
\leq C_1 \int_{E \times E} |S_n(\one_{A \times A})(x,y)|^{\alpha^{\prime}} (\mu \times \mu) (dx \; dy)
$$
$$
= C_1 (a_n^{\prime})^{\alpha^{\prime}} \mu(\varphi \leq n)^2 \int_{E \times E} \left| \frac{S_n(\one_{A \times A})(x,y)}{a_n^{\prime}} \right|^{\alpha^{\prime}} (\mu \times \mu)_n (dx \; dy)\,. 
$$
Because of the uniform integrability of $(|S_n(\one_{A \times A}) / a_n^{\prime}|^{\alpha^{\prime}}, n\geq 1)$ with respect to $(\mu \times \mu)_n$, we see that $\int_{E \times E} \left| S_n(\one_{A \times A}) / a_n^{\prime} \right|^{\alpha^{\prime}} d(\mu \times \mu)_n$ converges to some positive finite constant. The rest of the discussion is the same as Lemma \ref{l:L_alpha_conv_2}. 
\end{proof}

Finally we want to mention several inequalities that will be frequently used in the sequel. 

\begin{lemma} \label{l:UB_for_U1}
Fix $\xi>0$ as specified in Lemma \ref{l:L_alpha_conv_1}, \ref{l:L_alpha_conv_2}, or \ref{l:L_alpha_conv_3}. Let $\alpha^{\prime} = \alpha - \xi$ and define
$$
W_{ij}^{(n,\alpha^{\prime})} = \frac{1}{c_n} \sum_{k=1}^n f_k(V_i) f_{k+h}(V_j) q(V_i)^{-1/\alpha^{\prime}} q(V_j)^{-1/\alpha^{\prime}}\,.
$$
Let
$$
\ln_+ x = \begin{cases} \ln x & \text{if } x > 1, \\
0 & \text{otherwise}. 
\end{cases}
$$
(a) There exist an integer $m_0>0$ and constants $C>0$, $\gamma < \alpha^{\prime}$, such that for any $m \geq m_0$, 
$$
E \left| \sum_{m < i < j < \infty} \hspace{-10pt} \epsilon_i \epsilon_j U_{\alpha,1}^{\leftarrow} \left( \frac{\Gamma_i q(V_i)}{2} \right) U_{\alpha,1}^{\leftarrow} \left( \frac{\Gamma_j q(V_j)}{2} \right) \frac{1}{c_n} \sum_{k=1}^n f_k(V_i) f_{k+h}(V_j) \one_{\{ |W_{ij}^{(n,\alpha^{\prime})}|^{\alpha^{\prime}} \leq ij \}} \right|^{\alpha^{\prime}} 
$$
$$
\leq C \Bigl(E\bigl(|W_{ij}^{(n,\alpha^{\prime})}|^{\alpha^{\prime}}(1+\ln_+|W_{ij}^{(n,\alpha^{\prime})}|)\bigr)\Bigr)^{\gamma}, 
$$
$$
E \left| \sum_{m < i < j < \infty} \hspace{-10pt} \epsilon_i \epsilon_j U_{\alpha,1}^{\leftarrow} \left( \frac{\Gamma_i q(V_i)}{2} \right) U_{\alpha,1}^{\leftarrow} \left( \frac{\Gamma_j q(V_j)}{2} \right) \frac{1}{c_n} \sum_{k=1}^n f_k(V_i) f_{k+h}(V_j) \one_{\{ |W_{ij}^{(n,\alpha^{\prime})}|^{\alpha^{\prime}} > ij \}} \right|^{\alpha^{\prime}} 
$$
$$
\leq C E\bigl(|W_{ij}^{(n,\alpha^{\prime})}|^{\alpha^{\prime}}(1+\ln_+^2 |W_{ij}^{(n,\alpha^{\prime})}|)\bigr).
$$

\noindent (b) There exist an integer $m_0>0$ and constants $C>0$, $\gamma < \alpha^{\prime}$, such that for any $m \geq m_0$ and $i \geq 1$,
$$
E \left| \sum_{j=m+1}^{\infty} \epsilon_j U_{\alpha,1}^{\leftarrow} \left( \frac{\Gamma_j q(V_j)}{2} \right) \frac{1}{c_n} \sum_{k=1}^n f_k(V_i) f_{k+h}(V_j) q(V_i)^{-1/\alpha^{\prime}} \one_{\{ |W_{ij}^{(n,\alpha^{\prime})}|^{\alpha^{\prime}} \leq j \}} \right|^{\alpha^{\prime}} \leq C \bigl(E|W_{ij}^{(n,\alpha^{\prime})}|^{\alpha^{\prime}}\bigr)^{\gamma}\,,
$$
$$
E \left| \sum_{j=m+1}^{\infty} \epsilon_j U_{\alpha,1}^{\leftarrow} \left( \frac{\Gamma_j q(V_j)}{2} \right) \frac{1}{c_n} \sum_{k=1}^n f_k(V_i) f_{k+h}(V_j) q(V_i)^{-1/\alpha^{\prime}} \one_{\{ |W_{ij}^{(n,\alpha^{\prime})}|^{\alpha^{\prime}} > j \}} \right|^{\alpha^{\prime}} 
$$
$$
\leq C E\bigl(|W_{ij}^{(n,\alpha^{\prime})}|^{\alpha^{\prime}}(1+\ln_+|W_{ij}^{(n,\alpha^{\prime})}|)\bigr)\,.
$$
\end{lemma}

\begin{proof}
The proof is analogous to that of Proposition 5.1 in \cite{samorodnitsky:szulga:1989}, but an obvious upper bound $U_{\alpha,1}^{\leftarrow}(x) \leq C x^{-1/\alpha^{\prime}}$, $x>0$ has to be suitably applied. 
\end{proof}

\begin{remark}  \label{rk:light_tail_inequality}
The inequalities in Lemma \ref{l:UB_for_U1} still holds, even if the parameter $\alpha^{\prime}$ and the inverse function $U_{\alpha,1}^{\leftarrow}(\cdot)$ are replaced by the constant $p_0$ given in (\ref{e:lower_tail_Levy}) and $U_{\alpha,2}^{\leftarrow}(\cdot)$, respectively.
\end{remark}

\section{Proof of Theorem \ref{t:main_theorem1}}

Now, we are ready to prove Theorem \ref{t:main_theorem1}. Before embarking on the proof, however, it would be beneficial to describe the outline of the proof. First of all, Proposition \ref{p:growth_rate_cn} below provides information on the asymptotics of $(c_n)$. Proposition \ref{p:diagonal-offdiagonal} then splits the series representation of $\sum_{k=1}^n X_k^{(l)} X_{k+h}^{(l)}$, $l=1,2$ into a diagonal part 
\begin{equation}  \label{e:diagonal}
Y_{n,l}^{\prime}(h) = \sum_{i=1}^{\infty} U_{\alpha,l}^{\leftarrow} \left( \frac{\Gamma_i q(V_i)}{2} \right)^2 \sum_{k=1}^n f_k(V_i) f_{k+h}(V_i)
\end{equation}
and an off-diagonal part
\begin{equation}  \label{e:off.diagonal}
Y_{n,l}^{\prime \prime}(h) = \sum_{i \neq j} \epsilon_i \epsilon_j U_{\alpha,l}^{\leftarrow} \left( \frac{\Gamma_i q(V_i)}{2} \right) U_{\alpha,l}^{\leftarrow} \left( \frac{\Gamma_j q(V_j)}{2} \right) \sum_{k=1}^n f_k(V_i) f_{k+h}(V_j)\,,
\end{equation}
and further, the diagonal part is shown to have integral representation. Subsequently, Proposition \ref{p:prop_for_1stterm} verifies weak convergence
$$
\left( \frac{Y_{n,1}^{\prime}(h)}{c_n},\, h=0,\dots,H \right) \Rightarrow \bigl( \mu(f \cdot f_h) W,\, h=0, \dots, H \bigr) \ \ \ \text{in } \bbr^{H+1}\,,
$$ 
as $n \to \infty$, where $W$ is defined in \eqref{e:pos_stable_subordinator}. Propositions \ref{p:prop_for_2ndterm}, \ref{p:prop_for_3rdterm} and \ref{p:prop_for_4thterm} then prove the followings, respectively:
\begin{align*}
c_n^{-1} Y_{n,1}^{\prime \prime}(h) &\stackrel{p}{\to} 0\,, \ \ h=0,1,\dots,H\,, \\
c_n^{-1} Y_{n,2}^{\prime}(0) &\stackrel{p}{\to} 0\,, \ \ \text{and} \\
c_n^{-1} Y_{n,2}^{\prime \prime}(0) &\stackrel{p}{\to} 0\,.
\end{align*}
Therefore, because of the stationarity of the process $\BX=(X_1,X_2,\dots)$, the leading term on the right hand side of \eqref{e:decomp.4parts} turns out to be $c_n^{-1} Y_{n,1}^{\prime}(h)$ and all the others will vanish as $n \to \infty$; hence, the proof of Theorem \ref{t:main_theorem1} will be completed. 

\begin{proposition} \label{p:growth_rate_cn}
Under the assumptions of Theorem \ref{t:main_theorem1}, 
\begin{equation}  \label{e:main.target}
\left( \int_E |S_n(f^2)|^{\alpha/2} d\mu \right)^{2/\alpha} \sim \mu(f^2) C_{\alpha,\beta} a_n w_n^{2/\alpha}, \ \ \ \text{as } n \to \infty\,,
\end{equation}
and 
$$
\rho_{\alpha}\bigl( (c_n a_n^{-1})^{1/2},\infty \bigr) \sim 2^{-1} C_{\alpha/2} C_{\alpha, \beta}^{-\alpha/2} w_n^{-1} \ \ \text{as } n \to \infty\,,
$$
where $C_{\alpha,\beta}$ is a positive constant given by \eqref{e:def.C_al.be}, and $C_{\alpha/2}$ is a tail constant of an $\alpha/2$-stable random variable. 
\end{proposition}

\begin{proof}
The second asymptotic relation is easy to check from the definition of $(c_n)$, so in what follows, we only prove \eqref{e:main.target}. 
We write
$$
\left( \int_E |S_n(f^2)|^{\alpha/2} d\mu \right)^{2/\alpha} =  a_n \mu(\varphi \leq n)^{2/\alpha} \left( \int_E \left| \frac{S_n(f^2)}{a_n} \right|^{\alpha/2} d\mu_n \right)^{2/\alpha},
$$
where $\mu_n(\cdot) = \mu \bigl( \cdot \cap \{ \varphi \leq n \} \bigr) / \mu(\varphi \leq n)$. 
Because of uniform integrability of $( |S_n(f^2) / a_n|^{\alpha/2},\, n \geq 1 )$ with respect to $\mu_n$, Lemma \ref{l:DK_theorem} implies 
$$
\left( \int_E \left| \frac{S_n(f^2)}{a_n} \right|^{\alpha/2} d\mu_n \right)^{2/\alpha} \to \mu(f^2) C_{\alpha,\beta} \ \ \ \text{as } n \to \infty\,. 
$$
\end{proof}

\begin{proposition}  \label{p:diagonal-offdiagonal}
Under the assumptions of Theorem \ref{t:main_theorem1}, let $H \geq 0$, $n > 0$, and $l=1,2$. Then we write
$$
\left(  \sum_{k=1}^n X_k^{(l)} X_{k+h}^{(l)}, \, h=0,\dots,H \right) = \bigl( Y_{n,l}^{\prime}(h) + Y_{n,l}^{\prime \prime}(h),\,  h=0,\dots,H \bigr)
$$
with $Y_{n,l}^{\prime}(h)$ and $Y_{n,l}^{\prime \prime}(h)$ given in \eqref{e:diagonal} and \eqref{e:off.diagonal}.
Furthermore, $Y_{n,l}^{\prime}(h)$ is represented in law by 
\begin{equation}  \label{e:stoch.pos.rep}
Y_{n,l}^{\prime}(h) \stackrel{d}{=} \int_E \sum_{k=1}^n f_k(x) f_{k+h}(x) d \widetilde{M}_l(x)\,.
\end{equation}
Here, $\widetilde{M}_l$ is a positive \id\ random measure defined by
$$
E e^{iu \widetilde{M}_l (A)} = \exp \{ \mu(A) \int_{(0,\infty)} (e^{iux} - 1) \widetilde{\rho}_{\frac{\alpha}{2},l}(dx) \}\,, \ \ u \in \bbr\,, 
$$
where $\widetilde{\rho}_{\frac{\alpha}{2},l}$ is a local L\'{e}vy measure concentrated on the positive half-line such that
\begin{equation}  \label{e:positive_Levy}
\widetilde{\rho}_{\frac{\alpha}{2},l} (x,\infty) = 2 \rho_{\alpha,l}(x^{1/2},\infty) \ \ \ \text{for } x>0\,.  
\end{equation}
\end{proposition}

\begin{proof}
Clearly, it suffices to show \eqref{e:stoch.pos.rep}. In view of \cite{rosinski:1990}, we only check that
\begin{equation}  \label{e:fisrtRosinski}
\int_0^{\infty} P \left( U_{\alpha,l}^{\leftarrow} \left( \frac{r q(V_1)}{2} \right)^2 \sum_{k=1}^n f_k(V_1) f_{k+h}(V_1) \in \cdot \right) dr 
= (\widetilde{\rho}_{\frac{\alpha}{2}, l} \times \mu) \{ (v,x): v \sum_{k=1}^n f_k(x) f_{k+h}(x) \in \cdot  \}   
\end{equation}
and 
\begin{equation}   \label{e:secondRosinski}
\int_E \int_{\mathbb{R}} \min \left(1,  \left|v \sum_{k=1}^n f_k(x) f_{k+h}(x) \right| \right) \widetilde{\rho}_{\frac{\alpha}{2}, l} (dv) \mu(dx) < \infty\,. 
\end{equation}
Note that the right hand side of \eqref{e:fisrtRosinski} is exactly equal to the L\'{e}vy measure of $Y_{n,l}^{\prime}(h)$. 

Since a simple calculation verifies \eqref{e:fisrtRosinski}, we only prove \eqref{e:secondRosinski}. By regular variation of the local L\'{e}vy measure $\rho_{\alpha}$, the Potter bound (e.g., Proposition 0.8 in \cite{resnick:1987}) provides
$$
\widetilde{\rho}_{\frac{\alpha}{2},1}(x,\infty) \leq C_1 x^{-(\alpha-\xi)/2}, \ \ \  x> 0
$$
for some constants $0 < \xi < \alpha$ and $C_1>0$. Also by \eqref{e:lower_tail_Levy}, we get an obvious upper bound; for some $C_2>0$,
$$
\widetilde{\rho}_{\frac{\alpha}{2},2}(x,\infty) \leq C_2 x^{-p_0/2}, \ \ \ x > 0\,.
$$
These bounds, together with the fact that $f$ has a support of finite $\mu$-measure and $f \in L^2(\mu)$, can establish \eqref{e:secondRosinski}.
\end{proof}

\begin{proposition}  \label{p:prop_for_1stterm}
Under the assumptions of Theorem \ref{t:main_theorem1}, for any $H \geq 0$, we have as $n \to \infty$,
$$
\left( \frac{Y_{n,1}^{\prime}(h)}{c_n},\, h=0,\dots,H \right) \Rightarrow \bigl( \mu(f \cdot f_h) W, \, h=0, \dots, H \bigr) \ \ \ \text{in } \bbr^{H+1}\,,
$$
where $W$ is defined in \eqref{e:pos_stable_subordinator}.
\end{proposition}

\begin{proof}
By virtue of the Cramer-Wold device, we only have to show that as $n \to \infty$,
\begin{equation}  \label{e:CramerWold}
\frac{1}{c_n} \sum_{h=0}^H \theta_h Y_{n,1}^{\prime}(h) \Rightarrow \sum_{h=0}^H \theta_h \mu(f \cdot f_h) W \ \  \text{in } \mathbb{R} 
\end{equation}
for every $\theta_0, \dots, \theta_H \in \mathbb{R}$. Let $\phi(x) = f(x) \sum_{h=0}^H \theta_h f_h(x)$. Then, \eqref{e:CramerWold} is equivalent to 
\begin{equation}  \label{e:phiCramerWold}
\frac{1}{c_n} \int_E S_n(\phi)(x) d \widetilde{M}_1(x) \Rightarrow \mu(\phi) W \ \ \ \text{in } \mathbb{R}\,.
\end{equation}
A sufficient condition for weak convergence of the left hand side in \eqref{e:phiCramerWold} reduces to the following (e.g., Theorem 13.14 in \cite{kallenberg:1997}): for every $r>0$, as $n \to \infty$,
\begin{equation}  \label{e:firstKallen} 
\int_E \left( \frac{S_n(\phi)}{c_n} \right)^2 \int_0^{r c_n |S_n(\phi)|^{-1}} \hspace{-20pt} x \widetilde{\rho}_{\frac{\alpha}{2},1} (x,\infty) dx d\mu \to \frac{r^{2-\alpha/2}C_{\alpha/2}}{2-\alpha/2} | \mu(\phi)|^{\alpha/2}\,, 
\end{equation}
\begin{equation}  \label{e:secondKallen}
\int_E \widetilde{\rho}_{\frac{\alpha}{2},1} (rc_n |S_n(\phi)|^{-1}, \infty) d\mu \to r^{-\alpha/2} C_{\alpha/2} | \mu(\phi)|^{\alpha/2}\,,
\end{equation}
and
\begin{equation}  \label{e:thirdKallen}
\int_E \frac{S_n(\phi)}{c_n} \int_0^{c_n |S_n(\phi)|^{-1}} \hspace{-20pt} \widetilde{\rho}_{\frac{\alpha}{2},1} (x,\infty) dx d\mu \to \frac{2C_{\alpha/2}}{2-\alpha} sgn(\mu(\phi)) | \mu(\phi)|^{\alpha/2}
\end{equation}
($sgn(u)=u/|u|$ if $u \neq 0$ and $sgn(0) = 0$). 
We only prove \eqref{e:firstKallen}, because \eqref{e:secondKallen} and \eqref{e:thirdKallen} can be handled analogously.  

For \eqref{e:firstKallen}, we need to use the result of Lemma \ref{l:DK_theorem}: as $n \to \infty$,
$$
\frac{S_n(\phi)}{a_n} \Rightarrow \mu(\phi) \Gamma(1+\beta) M_{\beta}(1-V_{\beta}) \ \ \text{in } \bbr\,,
$$
where the weak convergence takes place under a probability measure $\mu_n(\cdot) = \mu(\cdot \cap \{ \varphi \leq n \}) / \mu(\varphi \leq n)$. $( M_{\beta}(t) )$ is the Mittag-Leffler process defined on some probability space $(\Omega^{\prime}, \mathcal{F}^{\prime}, P^{\prime})$ (the definition is given in \eqref{e:MLprocess}), and $V_{\beta}$ is a random variable defined on the same probability space with density given by \eqref{e:pdfTinf}. Here, $M_{\beta}(t)$ and $V_{\beta}$ are independent. 

Applying the Skorohod's embedding theorem, there exist random variables $Y$ and $Y_n$, $n=1,2,\dots$ defined on some probability space $(\Omega^*,\mathcal{F}^*,P^*)$ such that
\begin{align*}
P^* \circ Y_n^{-1} &= \mu_n \circ \left( \frac{S_n(\phi)}{a_n} \right)^{-1}, \ \ \ n=1,2,\dots, \\
P^* \circ Y^{-1} &= P^{\prime} \circ \bigl(\mu(\phi) \Gamma(1+\beta) M_{\beta}(1-V_{\beta})\bigr)^{-1}\,, \\
Y_n &\to Y \ \ \text{as } n \to \infty, \ P^* \text{-a.s.}.
\end{align*}
Let $\psi(y) = y^{-2} \int_0^{ry} x \widetilde{\rho}_{\frac{\alpha}{2},1}(x,\infty) dx$, then we can proceed
$$
\int_E \left( \frac{S_n(\phi)}{c_n} \right)^2 \int_0^{rc_n |S_n(\phi)|^{-1}} \hspace{-20pt} x \widetilde{\rho}_{\frac{\alpha}{2},1} (x,\infty) dx d\mu = \int_E \psi \left( \frac{c_n}{|S_n(\phi)|} \right) d\mu 
$$
$$
= \mu(\varphi \leq n) E^* \Bigl[ \psi \Bigl( \frac{c_n}{a_n |Y_n|} \Bigr) \Bigr]. 
$$
It follows from \eqref{e:RV_index_cn} that $c_n a_n^{-1} |Y_n|^{-1} \to \infty$, $P^*$-a.s.. Therefore, Karamata's theorem (e.g., Theorem 0.6 in \cite{resnick:1987}) yields 
$$
\psi \left( \frac{c_n}{a_n |Y_n|} \right) \sim \frac{r^2}{2-\alpha/2} \widetilde{\rho}_{\frac{\alpha}{2},1} (rc_n a_n^{-1} |Y_n|^{-1}, \infty) \ \ \  \text{as } n \to \infty, \ P^* \text{-a.s.}
$$
From uniform convergence theorem of regularly varying functions of negative indices (e.g., Proposition 0.5 in \cite{resnick:1987}), we can say that
$$
\widetilde{\rho}_{\frac{\alpha}{2},1} (rc_n a_n^{-1} |Y_n|^{-1}, \infty) \sim r^{-\alpha/2} |Y_n|^{\alpha/2} \widetilde{\rho}_{\frac{\alpha}{2},1} (c_n a_n^{-1}, \infty) \ \ \ \text{as } n \to \infty, \ \ P^* \text{-a.s.}
$$
From Proposition \ref{p:growth_rate_cn} and \eqref{e:positive_Levy}, 
$$
\mu(\varphi \leq n) \psi \left( \frac{c_n}{a_n |Y_n|} \right) \sim \frac{r^{2-\alpha/2}}{2-\alpha/2} \mu(\varphi \leq n) |Y_n|^{\alpha/2} \widetilde{\rho}_{\frac{\alpha}{2},1} (c_n a_n^{-1}, \infty) 
$$
$$ 
\to \frac{r^{2-\alpha/2}C_{\alpha/2}}{2-\alpha/2} C_{\alpha,\beta}^{-\alpha/2} |Y|^{\alpha/2} \ \ \ \text{as } n \to \infty, \ \ \ P^* \text{-a.s.} 
$$
Integrating the limit yields
$$
E^* \left[ \frac{r^{2-\alpha/2}C_{\alpha/2}}{2-\alpha/2} C_{\alpha,\beta}^{-\alpha/2}  |Y|^{\alpha/2} \right] = \frac{r^{2-\alpha/2}C_{\alpha/2}}{2-\alpha/2} |\mu(\phi)|^{\alpha/2}\,,
$$
which is exactly the right hand side of \eqref{e:firstKallen}. 
Now, to finish the proof, we need to justify taking the limit under the integral. For this, we will apply the so-called Pratt's lemma (see \cite{pratt:1960}). According to Pratt's lemma, we must find a sequence of measurable functions $G_0, G_1, \dots$ defined on $(\Omega^*,\mathcal{F}^*,P^*)$ such that 
\begin{align}
\mu(\varphi \leq n) \psi \left( \frac{c_n}{a_n |Y_n|} \right) &\leq G_n \ \ \ P^* \text{-a.s.}, \ \ \ n=1,2,\dots, \label{e:firstPratt} \\
G_n &\to G_0 \ \ \ \text{as } n \to \infty \ \ \ P^* \text{-a.s.}, \; \text{and} \label{e:secondPratt} \\
E^* G_n &\to E^* G_0 \ \ \ \text{as } n \to \infty\,. \label{e:thirdPratt}
\end{align}
For \eqref{e:firstPratt}, there is a $C_1>0$ such that
$$
\mu(\varphi \leq n) \psi \left( \frac{c_n}{a_n |Y_n|} \right) \leq C_1 \frac{\psi(c_n a_n^{-1} |Y_n|^{-1})}{\psi(c_n a_n^{-1})}
$$
because $\mu(\varphi \leq n) \psi(c_n a_n^{-1})$ has a positive and finite limit. \\
Applying the Potter bound, for any fixed $0 < \xi < \min(\alpha, 2-\alpha)$, we have
$$
\frac{\psi(c_n a_n^{-1} |Y_n|^{-1})}{\psi(c_n a_n^{-1})} \one_{\{ c_n > a_n |Y_n| \}} \leq C_2 (|Y_n|^{(\alpha-\xi)/2} + |Y_n|^{(\alpha+\xi)/2})
$$
for some $C_2>0$. \\
Since $\psi$ is bounded on $(0,1]$, for some constant $C_3 \geq C_2$, 
$$
\frac{\psi(c_n a_n^{-1} |Y_n|^{-1})}{\psi(c_n a_n^{-1})} \one_{\{ c_n \leq a_n |Y_n| \}} \leq \frac{C_3}{\psi(c_n a_n^{-1})} \frac{a_n}{c_n} |Y_n|\,.
$$
Therefore, we may write
$$
\mu(\varphi \leq n) \psi \left( \frac{c_n}{a_n |Y_n|} \right) \leq C_3 \left( |Y_n|^{(\alpha-\xi)/2} + |Y_n|^{(\alpha+\xi)/2} + \frac{a_n}{c_n} \frac{|Y_n|}{\psi(c_n a_n^{-1})} \right)\,.
$$
Now, \eqref{e:firstPratt} is obtained by taking \\
$$
G_n = C_3 \left( |Y_n|^{(\alpha-\xi)/2} + |Y_n|^{(\alpha+\xi)/2} + \frac{a_n}{c_n} \frac{|Y_n|}{\psi(c_n a_n^{-1})} \right), \ \ \ n=1,2,\dots.
$$
Let
$$
G_0 = C_3 \left( |Y|^{(\alpha-\xi)/2} + |Y|^{(\alpha+\xi)/2} \right)\,.
$$
We know that $a_n c_n^{-1} \in RV_{-2(1-\beta)/\alpha}$ and $\psi(c_n a_n^{-1}) \in RV_{\beta-1}$; thus, 
$$
\frac{a_n}{c_n} \frac{1}{\psi(c_n a_n^{-1})} \to 0 \ \ \ \text{as } n \to \infty
$$
from which \eqref{e:secondPratt} follows. 

To show \eqref{e:thirdPratt}, recall that $\sup_{n \geq 1} E^*|Y_n| < \infty$ (see the proof of Proposition \ref{p:growth_rate_cn}). 
Thus, $( |Y_n|^{(\alpha \pm \xi)/2}, \, n \geq 1 )$ is uniformly integrable with respect to $P^*$, which in turn implies \eqref{e:thirdPratt}. Now, Pratt's lemma is applicable and \eqref{e:firstKallen} is complete.  
\end{proof}

\begin{proposition} \label{p:prop_for_2ndterm}
Under the assumptions of Theorem \ref{t:main_theorem1}, as $n \to \infty$,
$$
\frac{1}{c_n} Y_{n,1}^{\prime \prime}(h) = \sum_{i \neq j} \epsilon_i \epsilon_j U_{\alpha,1}^{\leftarrow} \left( \frac{\Gamma_i q(V_i)}{2} \right) U_{\alpha,1}^{\leftarrow} \left( \frac{\Gamma_j q(V_j)}{2} \right) \frac{1}{c_n} \sum_{k=1}^n f_k(V_i) f_{k+h} (V_j) \stackrel{p}{\to} 0\,, \ \ \ h = 0,1,2\dots. 
$$
\end{proposition}

\begin{proof}
Choose $\xi > 0$ as specified in Lemma \ref{l:L_alpha_conv_1} if $(\alpha,\beta)$ lies in the range \eqref{e:first_range}, and otherwise, choose $\xi> 0$ as specified in Lemma \ref{l:L_alpha_conv_2} or \ref{l:L_alpha_conv_3}. In either case, let $\alpha^{\prime} = \alpha - \xi$. 
For $i \neq j$, we set
$$
W_{ij}^{(n, \alpha^{\prime})} = \frac{1}{c_n} \sum_{k=1}^n f_k(V_i) f_{k+h}(V_j) q(V_i)^{-1/\alpha^{\prime}} q(V_j)^{-1/\alpha^{\prime}}\,.
$$
Note that these three lemmas have simultaneously demonstrated 
\begin{equation}  \label{e:needtwodimdiscussion_mb}
E \left| W_{ij}^{(n, \alpha^{\prime})} \right|^{\alpha^{\prime}} \to 0\,, \ \ \ \text{as } n \to \infty \ \textrm{for } i \neq j\,.
\end{equation}
In the sequel, we will basically follow the argument in Proposition 4.3 of \cite{resnick:samorodnitsky:xue:1999}. Denote
$$
\widetilde{W}_{ij}^{(n)} = \epsilon_i \epsilon_j U_{\alpha,1}^{\leftarrow} \left( \frac{\Gamma_i q(V_i)}{2} \right) U_{\alpha,1}^{\leftarrow} \left( \frac{\Gamma_j q(V_j)}{2} \right) \frac{1}{c_n} \sum_{k=1}^n f_k(V_i) f_{k+h} (V_j) \ \ \text{for } i \neq j.
$$
Owing to symmetry of the doubly infinite sum, we only have to show that $\sum_{i < j} \widetilde{W}_{ij}^{(n)} \stackrel{p}{\to} 0$. According to Lemma \ref{l:UB_for_U1}, there exist an integer $m_0$ and constants $C > 0$ and $\gamma < \alpha^{\prime}$ such that for any $m \geq m_0$, all the inequalities given in (a) and (b) of that lemma hold.  

Next, we decompose $\sum_{i < j} \widetilde{W}_{ij}^{(n)}$ into three summands
$$
\sum_{i < j} \widetilde{W}_{ij}^{(n)} = \sum_{i=1}^{m_0} \sum_{j=i+1}^{m_0} \widetilde{W}_{ij}^{(n)} + \sum_{i=1}^{m_0} \sum_{j=m_0+1}^{\infty} \widetilde{W}_{ij}^{(n)} + \sum_{m_0 < i < j < \infty} \widetilde{W}_{ij}^{(n)}\,.
$$
Now, we only need to prove the following: as $n \to \infty$,
$$
(i): \, \widetilde{W}_{ij}^{(n)} \stackrel{p}{\to} 0 \ \ \text{for all } i, j;
$$
$$
(ii):  \sum_{j=m_0+1}^{\infty} \widetilde{W}_{ij}^{(n)} \stackrel{p}{\to} 0 \ \ \text{for all } i; 
$$
$$
(iii): \sum_{m_0 < i < j < \infty} \widetilde{W}_{ij}^{(n)} \stackrel{p}{\to} 0\,.
$$
By the bound $U_{\alpha,1}^{\leftarrow}(x) < C x^{-1/\alpha^{\prime}}$ and \eqref{e:needtwodimdiscussion_mb}, it is evident that $\widetilde{W}_{ij}^{(n)}$ converges to $0$ in probability, which proves (i). For $(ii)$ and $(iii)$, by virtue of the inequalities given in Lemma \ref{l:UB_for_U1}, it suffices to show that
$$
E(|W_{ij}^{(n, \alpha^{\prime})}|^{\alpha^{\prime}} (1 + \ln_+^2|W_{ij}^{(n, \alpha^{\prime})}|)) \to 0\,, \ \ \ \text{as } n \to \infty, \ i \neq j\,.
$$
To show this, let
$$
B^{(n, \alpha^{\prime})}(x,y) = \frac{1}{c_n} \sum_{k=1}^n f_k(x) f_{k+h}(y) q(x)^{-1/\alpha^{\prime}} q(y)^{-1/\alpha^{\prime}}. 
$$
Here, it is important to note that the choice of the density $q$ does not affect the distribution of $Y_{n,1}^{\prime \prime}(h)$; therefore, we can particularly take 
$$
q(x) = Q(x) \left( \int_E Q(u) d\mu \right)^{-1},
$$
where 
$$
Q(x) = \max \Bigl( q_0(x), \bigl(\sum_{k=1}^{n+h} f_k(x)^2 \bigr)^{\alpha^{\prime}/2} \Bigr)\,.
$$
Here, $q_0:E \to (0,\infty)$ is an arbitrarily selected, strictly positive density. \\
By the Cauchy-Schwarz inequality, 
$$
\sup_{x,y \in E} |B^{(n,\alpha^{\prime})}(x,y)| \leq \sup_{x,y \in E} \frac{1}{c_n} \bigl(\sum_{k=1}^n f_k(x)^2 \bigr)^{1/2} \bigl(\sum_{k=1}^n f_{k+h}(y)^2 \bigr)^{1/2} q(x)^{-1/\alpha^{\prime}} q(y)^{-1/\alpha^{\prime}} 
$$
$$
\leq \frac{1}{c_n} \Bigl( 1 + \int_E (\sum_{k=1}^{n+h} f_k(x)^2)^{\alpha^{\prime}/2} d\mu \Bigr)^{2/\alpha^{\prime}} \in RV_{2(1-\beta)(\frac{1}{\alpha^{\prime}} - \frac{1}{\alpha})}\,,
$$
where the last regular variation index is obtained from \eqref{e:RV_index_cn}. \\
Now, we have
$$
E\bigl(|W_{ij}^{(n,\alpha^{\prime})}|^{\alpha^{\prime}} (1 + \ln_+^2|W_{ij}^{(n,\alpha^{\prime})}|)\bigr) \leq \bigl(1 + \ln_+^2 \sup_{x,y \in E} |B^{(n,\alpha^{\prime})}(x,y)| \bigr) E|W_{ij}^{(n,\alpha^{\prime})}|^{\alpha^{\prime}}\,.
$$
Observe that $E|W_{ij}^{(n,\alpha^{\prime})}|^{\alpha^{\prime}}$ has a negative regular variation exponent (see the proofs of Lemmas \ref{l:L_alpha_conv_1}, \ref{l:L_alpha_conv_2}, and \ref{l:L_alpha_conv_3}), and hence, the right hand side vanishes as $n \to \infty$.
\end{proof}

\begin{proposition} \label{p:prop_for_3rdterm}
Under the assumptions of Theorem \ref{t:main_theorem1},
$$
\frac{1}{c_n} Y_{n,2}^{\prime}(0) \stackrel{d}{=} \frac{1}{c_n} \int_E S_n(f^2)(x) d \widetilde{M}_2(x) \stackrel{p}{\to} 0 \ \ \ \text{as } n \to \infty\,.
$$
\end{proposition}

\begin{proof}
From the standard argument for convergence in law of the sequence of infinitely divisible random variables (e.g., Theorem 13.14 in \cite{kallenberg:1997}), we only have to check that as $n \to \infty$,
$$
\int_E \left( \frac{S_n(f^2)}{c_n} \right)^2 \int_0^{c_nS_n(f^2)^{-1}} \hspace{-20pt} x \widetilde{\rho}_{\frac{\alpha}{2}, 2}(x,\infty) dx d\mu \to 0\,,
$$
$$
\int_E \widetilde{\rho}_{\frac{\alpha}{2},2}(c_nS_n(f^2)^{-1}, \infty) d\mu \to 0\,, 
$$
and
$$
\int_E  \frac{S_n(f^2)}{c_n} \int_0^{c_nS_n(f^2)^{-1}} \hspace{-20pt} \widetilde{\rho}_{\frac{\alpha}{2}, 2}(x,\infty) dx d\mu \to 0\,.
$$
An obvious upper bound $\widetilde{\rho}_{\frac{\alpha}{2}, 2}(x,\infty) \leq C x^{-p_0/2}$, $x > 0$, and the integrability condition $f \in L^2(\mu)$ easily prove these limits. 
\end{proof}

\begin{proposition} \label{p:prop_for_4thterm}
Under the assumptions of Theorem \ref{t:main_theorem1},
$$
\frac{1}{c_n} Y_{n,2}^{\prime \prime}(0) = \sum_{i \neq j} \epsilon_i \epsilon_j U_{\alpha,2}^{\leftarrow} \left( \frac{\Gamma_i q(V_i)}{2} \right) U_{\alpha,2}^{\leftarrow} \left( \frac{\Gamma_j q(V_j)}{2} \right) \frac{1}{c_n} \sum_{k=1}^n f_k(V_i) f_k (V_j) \stackrel{p}{\to} 0 \ \ \ \text{as } n \to \infty\,.
$$
\end{proposition}

\begin{proof}
The proof is analogous to that of Proposition \ref{p:prop_for_2ndterm}. Taking advantage of the inequalities given in Lemma \ref{l:UB_for_U1} (see also Remark \ref{rk:light_tail_inequality}), the proof will be finished if
$$
E\bigl(|W_{ij}^{(n,p_0)}|^{p_0} (1 + \ln_+^2 |W_{ij}^{(n,p_0)}|)\bigr) \to 0\,, \ \ \ \text{as } n \to \infty, \ i \neq j\,.
$$
The argument for showing this is mostly the same as in Proposition \ref{p:prop_for_2ndterm}, so we omit it. 
\end{proof}

\bigskip

{\bf Acknowledgement} First of all, the author should like to express his gratitude to professor Gennady Samorodnitsky for his discussion and helpful comments throughout this research. The author is also extremely grateful to two anonymous referees and an anonymous Associate Editor whose comments led to a substantial improvement of the presentation of the paper.

\end{document}